\documentclass[12pt]{extarticle}
\usepackage{amsmath, amsthm, amssymb, mathtools, hyperref,color}
\usepackage{graphicx}
\usepackage{caption}
\usepackage{subcaption}
\usepackage{tikz}
\usetikzlibrary{calc}

\newtheorem{theorem}{Theorem}
\numberwithin{theorem}{section}
\newtheorem{proposition}[theorem]{Proposition}
\newtheorem{lemma}[theorem]{Lemma}
\newtheorem{corollary}[theorem]{Corollary}
\newtheorem{definition}[theorem]{Definition}
\newtheorem{remark}[theorem]{Remark}
\newtheorem{conjecture}[theorem]{Conjecture}

\newtheorem{example}[theorem]{Example}

\begin{document}
{
\title{Brill-Noether Existence on Graphs via $\mathbb{R}$-Divisors, Polytopes and Lattices\\ \normalsize{Dedicated to Bernd Sturmfels on his 60th Birthday} }
\author{Madhusudan Manjunath\footnote {Part of this work was carried out while the author was visiting IHES, Bures-sur-Yvette and ICTP, Trieste. We thank the generous support and the warm hospitality of these institutes.The author was supported by a MATRICS grant (MTR/2019/000462) of the Department of Science and Technology (DST), India during the course of this work.}}
\maketitle

\begin{abstract}

We study Brill-Noether existence on a finite graph using methods from polyhedral geometry and lattices. We start by formulating analogues of the Brill-Noether conjectures (both the existence and non-existence parts) for $\mathbb{R}$-divisors, i.e. divisors with real coefficients, on a graph.  We then reformulate the Brill-Noether existence conjecture for $\mathbb{R}$-divisors on a graph in geometric terms, that we refer to as the covering radius conjecture and we show a weak version, in support of it. Using this, we show an approximate version of the Brill-Noether existence conjecture for divisors on a graph. As applications, we derive upper bounds on the gonality of a graph and its $\mathbb{R}$-divisor analogue. 



\end{abstract}
\section{Introduction}

Analogies between finite graphs and compact Riemann surfaces are at the interface of several branches of mathematics. For instance, combinatorics, algebraic geometry, complex analysis and probability theory.  Some instances of such analogies are in the work of Sunada on discrete geometric analysis \cite{Sun08}, Bacher, De La Harpe and Nagnibeda \cite{BacLarNag97} and Smirnov on discrete complex analsis \cite{Smi10}.  Baker and Norine in 2007 \cite{BakNor07} took a major step in this direction by developing a Riemann-Roch theory for graphs. This theory is closely related to Riemann-Roch theory in tropical algebraic geometry \cite{GatKer08} and has since inspired numerous applications \cite{CooDraPayRob12}, \cite{Ami14}.
 
In a follow-up to this work on the Riemann-Roch theorem, Baker \cite{Bak08} proposed an analogue of Brill-Noether theory for graphs.  The main goal of Brill-Noether theory is to classify divisors of prescribed degree and rank on a given object (a Riemann surface, an algebraic curve or a graph).  We start by stating one of the two key theorems in the Brill-Noether theory of algebraic curves. Let $X$ be a smooth, proper algebraic curve of genus $g$. Define $\rho(g,r,d)=g-(r+1)(g-d+r)$.

\begin{theorem}{\rm{(\bf Brill-Noether Existence for Curves)}}\label{bnexalg_theo}\cite{Kem71}, \cite{KleLak72}
For any pair of non-negative integers $(r,d)$ such that  $d \leq 2g-2$. If $\rho(g,r,d) \geq 0$, then there exists a divisor on $X$ with degree at most $d$ and rank equal to $r$. 
\end{theorem}

Furthermore,  if  $\rho(g,r,d) \geq 0$ and $r \geq d-g$, then $\rho(g,r,d)$ is a lower bound for the dimension of the space $W^r_d$ of linear series of degree $d$ and rank at least $r$ on the algebraic curve, see \cite[Chapter 5]{ArbCorGriHar85} for more details.
 The converse to Brill-Noether existence holds (only) for a general algebraic curve and is the content of the non-existence part of Brill-Noether theory \cite{GriHar80}.


Throughout the paper,  by $G$ we denote an undirected, connected, loop-free, multigraph with $n \geq 2$ vertices, $m$ edges and genus $g=m-n+1 \geq 1$.  Baker formulated analogues of both the existence and non-existence parts of the Brill-Noether theorem. In the existence direction, Baker \cite[Conjecture 3.9, Part 1]{Bak08} conjectured the following.

\begin{conjecture}\label{bakconjv1} {\rm{(\bf Brill-Noether Existence for Graphs)}} Fix two non-negative integers $r,~d$ such that $d \leq 2g-2$. If $\rho(g,r,d) \geq 0$, then there exists a divisor $D$ of degree at most $d$ and rank equal to $r$ on $G$.   \end{conjecture}

Both Theorem \ref{bnexalg_theo} and Conjecture \ref{bakconjv1} hold for $d>2g-2$ and are an immediate consequence of the corresponding Riemann-Roch theorems. The core of Brill-Noether theory is in the range $2g-2 \geq d \geq 0$ and we will focus on this for the rest of the paper.  

We refer to \cite{Bak08}, \cite{CooDraPayRob12} for Brill-Noether non-existence for graphs. The Brill-Noether existence conjecture for graphs remains wide open except for graphs of small genus. Atanasov and Ranganathan, 2018 \cite{AtaRan18} proved the conjecture for graphs of genus at most 5 and constructed families of graphs with increasing genus where the existence conjecture holds in rank one.  In a related direction,  Cools and Panizzut 2017 \cite{CooPan17} computed the gonality sequence of complete graphs. Their work builds on the algorithm of Cori and Le Borgne \cite{CorLeb16} for ranks of divisors on complete graphs.  Panizzut 2017 \cite{Pan17} studied the gonality of  graphs obtained from the complete graph $K_n$ by either deleting a small number of edges (at most $n-2$) or by deleting the edges forming a complete (sub)graph $K_h$ for $h<n$.

Baker \cite{Bak08} also considered analogues of the existence and non-existence theorems for metric graphs (also known as abstract tropical curves). Roughly speaking, a metric graph is a  graph with non-negative real edge lengths. For metric graphs, both the existence and non-existence parts are better understood. In the same paper, Baker deduced an analogue of the Brill-Noether theorem for metric graphs and tropical curves) \cite[Theorem 3.20]{Bak08}  using the corresponding theorem for algebraic curves (Theorem \ref{bnexalg_theo}) and the specialisation lemma \cite[Lemma 2.8]{Bak08}.  An important difference between graphs and metric graphs is that divisors on metric graphs can have support at non-vertices, i.e. at the interior of the edges whereas divisors on graphs can only have support at the vertices. This creates  a complication in applying a similar method to proving Brill-Noether existence for graphs. Even in the case of metric graphs a ``combinatorial" proof (that does not rely on the existence theorem for curves) would be of significant interest \cite[Remark 3.13]{Bak08}.  We refer to the recent work of Draisma and Vargas \cite{DraVer19} for a combinatorial proof of the case $r=1$ of Brill-Noether existence for metric graphs. 
Caporaso \cite{Cap11} claimed a proof of the existence conjecture for graphs using algebro-geometric methods. However, this proof is known to contain a gap, see \cite[Footnote 5]{BakJen16}.


 Baker also made a conjecture on the non-existence analogue for metric graphs that was proven by Cools, Draisma, Payne and Robeva \cite{CooDraPayRob12} and using this, they gave an alternative proof of the corresponding statement for algebraic curves.  


A key change in perspective in the current work is to study analogues of Brill-Noether theory, particularly the existence part, for $\mathbb{R}$-divisors on graphs, i.e. divisors with real coefficients, on a graph. We posit that this analogue is both an interesting topic in its own right and is also a useful tool to tackle the existence conjecture for graphs. This analogue for $\mathbb{R}$-divisors  allows the use of geometric methods similar in flavour to the geometry of numbers \cite{Cas71} for Brill-Noether type questions on graphs. We start by formulating a version of the existence conjecture for $\mathbb{R}$-divisors. 


\begin{conjecture}\label{bakconjreal_intro} {\rm (\bf {Brill-Noether Existence for  $\mathbb{R}$-Divisors on Graphs})} Let $\rho(g,r,d)=g-(r+1)(g-d+r)$. Fix two non-negative real numbers $r,~d$  such that $d \leq 2g-2$.  If $\rho(g,r,d) \geq 0$, then there exists an $\mathbb{R}$-divisor of degree at most $d$ and rank equal to $r$ on $G$.   \end{conjecture}

\begin{remark}
\rm
We conjecture a stronger version that if in addition $g-d+r \geq 1$, then $G$ has an $\mathbb{R}$-divisor of degree equal to $d$ and rank equal to $r$. We refer to Subsection \ref{bnreal_subsect} for more details.  \qed
\end{remark}

\begin{remark}

\rm
The study of $\mathbb{R}$-divisors on an algebraic variety is an important topic in the positivity aspects of algebraic geometry, we refer to Lazarsfeld's book \cite[Chapter 1]{Laz04} for more details. For both algebraic varieties and graphs, an advantage of $\mathbb{R}$-divisors is that they allow for perturbation arguments.  Note that the notion of $\mathbb{Q}$-divisor in \cite[Section 1]{GatKer08} (also known as $\mathbb{Q}$-rational divisor in \cite[Subsection 1D]{Bak08})  on a $\mathbb{Q}$-metric graph  is different in spirit from $\mathbb{R}$-divisors in this paper: a $\mathbb{Q}$-divisor  is a divisor supported at the $\mathbb{Q}$-rational points of the metric graph. \qed
\end{remark}

Let $\mathcal{N}_G$ be the set of non-special divisors on $G$, i.e. divisors of degree $g-1$ and rank minus one.  In the following, we reformulate the existence conjecture for $\mathbb{R}$-divisors on graphs in terms of a lower bound on the covering radius of the set $\mathcal{N}_G$ with respect to a certain family of polytopes $P_{1,\lambda}$.

\begin{conjecture}\label{geomverconj_intro} {\rm(\bf{Covering Radius Conjecture})} Let $\lambda \in [1/g,g]$ (recall that $g \geq 1$). The covering radius of $\mathcal{N}_G$ with respect to the polytope $P_{1,\lambda}$ is at least $\sqrt{g/\lambda}/n$ where $n$ is the number of vertices of $G$.  \end{conjecture}

For a real number $d$,  let $H_d$ be the  affine hyperplane $\{(x_1,\dots,x_n) \in \mathbb{R}^n|~\sum_{i=1}^{n}x_i=d\}$.  The set $\mathcal{N}_G$ can be regarded as a subset of $H_{g-1}$. Furthermore,  $\mathcal{N}_G$ carries the action of a full rank sublattice of $H_0$ and this action is given by translation.  This full rank sublattice is the Laplacian lattice of $G$, which is the group of all divisors on $G$ that are linearly equivalent to zero or equivalently, it is the image of the Laplacian matrix of $G$ over $\mathbb{Z}^n$. 

The polytope $P_{1,\lambda}$ is defined as the Minkowski sum $\triangle+\lambda \cdot \bar{\triangle}$ of two simplices $\triangle$ and $\bar{\triangle}$ that we refer to as the \emph{standard simplices} \cite{AmiMan10}. They are regular simplices of dimension $n-1$ with centroid at the origin and $\bar{\triangle}=-\triangle$.  A convex body $C$ contained in the hyperplane $H_0$ induces a distance function $d_C$ on it. Given any subset $T$ of $\mathbb{R}^n$ contained in $H_d$ that carries the translation action of a full rank sublattice of $H_0$, we can define the covering radius of $T$ with respect to $C$ that, roughly speaking, measures the distance from a point in $H_d$ to $T$ that is farthest to it with respect to the distance function $d_{C}$.
We refer to Section \ref{covrad_sect} for their definitions.


\begin{figure}[ht]
  \centering
  \begin{tikzpicture}
    \coordinate (Origin)   at (0,0);
    \coordinate (YAxisMax) at (0,5);

    \clip (-3,-2) rectangle (7cm,7cm); 
    
  \def\s{1.4}
  
    \def\q{1.7}

  \pgftransformcm{0.35*\s*\q}{-0.05*\s}{-0.2*\s*\q}{\s}{\pgfpoint{0.25*\s*\q cm}{0.35*\s cm}}
   \coordinate (Bone) at (0,2);
    \coordinate (Btwo) at (2,-2);
      \def\r{1.0}
          
   \foreach \x in {-10,-9,-8,...,9,10}{
      \foreach \y in {-10,-9,-8,...,9,10}{
        \node[draw, circle,inner sep=1.5pt,fill] at (2*\x,2*\y) {};

    \draw[black, ultra thick](0.5*\r*\s+2*\x,0.5*\r*\s+2*\y)--(2*\x,\r*\s+2*\y)--(-0.5*\r*\s+2*\x,0.5*\r*\s+2*\y)--(-0.5*\r*\s+2*\x,-0.5*\r*\s+2*\y)--(2*\x,-\r*\s+2*\y)--(0.5*\r*\s+2*\x,-0.5*\r*\s+2*\y)-- cycle;

      }

    }

 \pgftransformcm{1}{0}{0}{1}{\pgfpoint{-0.2*\s*\q cm}{0.2*\s cm}}
   \coordinate (Bone) at (0,2);
    \coordinate (Btwo) at (2,-2);
   \foreach \x in {-10,-9,...,10}{
      \foreach \y in {-10,-9,...,10}{
        \node[draw, rectangle,inner sep=1.5pt,fill] at (2*\x,2*\y) {};

    \draw[black, ultra thick](0.5*\r*\s+2*\x,0.5*\r*\s+2*\y)--(2*\x,\r*\s+2*\y)--(-0.5*\r*\s+2*\x,0.5*\r*\s+2*\y)--(-0.5*\r*\s+2*\x,-0.5*\r*\s+2*\y)--(2*\x,-\r*\s+2*\y)--(0.5*\r*\s+2*\x,-0.5*\r*\s+2*\y)-- cycle;

      }

    }

  \end{tikzpicture}
  \caption{An illustration of an arrangement underlying the covering radius conjecture.}
  \label{covradconj_fig}
\end{figure}

  Figure \ref{covradconj_fig} illustrates the covering radius conjecture in the case of a multigraph on three vertices $v_1,~v_2$ and $v_3$ with three, eight and four edges between $(v_1,v_2)$, $(v_2,v_3)$ and $(v_1,v_3)$, respectively, and $\lambda=1$. 
In this case, $\mathcal{N}_G$ is contained in the affine hyperplane $H_{12}$ in $\mathbb{R}^3$ (since $g=13$ and $n=3$). The circular and square dots represent the two divisor classes in $\mathcal{N}_G$. We have an arrangement of (regular) hexagons 
centred at these points. The covering radius of this arrangement is the minimum dilation of the regular hexagons such that the arrangement covers the plane. Conjecture \ref{geomverconj_intro} asserts that it is at least $\sqrt{g}/n$, in this case $\sqrt{13}/3$. We show that Brill-Noether existence for $\mathbb{R}$-divisors and the covering radius conjecture are equivalent. We refer to Section \ref{bnreal_sect} for more details. 

 

The covering radius conjecture (and hence, Brill-Noether existence for $\mathbb{R}$-divisors) is wide open. As a first step, we prove a weaker version of the existence conjecture for arbitrary (undirected, connected) multigraphs in terms of a parameter called \emph{stretch factor}.   An undirected, connected, multigraph $G$ is called \emph{dense} if $m \geq n(n-1)$ (in other words, its genus $g \geq (n-1)^2$).  The stretch factor $\kappa$ of $G$ is defined to be the maximum of $\lceil(n(n-1))/m\rceil$ and one.

 Let $\beta$ be a positive integer.  By the \emph{uniform thickening} of $G$ by the factor $\beta$, we mean the graph on the same set of vertices as $G$ and with $\beta \cdot m(u,v)$ edges between $u$ and $v$, where $m(u,v)$ is the number of edges between $u$ and $v$ in $G$, see Example \ref{multikn_ex}.  The stretch factor measures the minimum uniform thickening of the edges of $G$ so that the resulting multigraph becomes dense. Hence, $\kappa=1$ if and only if the multigraph is dense and  the stretch factor of any connected multigraph is at most $n$.

\begin{theorem}\label{bnreal_intro}
Consider any undirected, connected, multigraph $G$ of genus $g$ and stretch factor $\kappa$.  Let $r$ and $d$ be non-negative real numbers such that $d \leq 2g-2$. If $\tilde{\rho}(g,r,d)=g-\kappa(r+1)(g-d+r) \geq 0$, then there exists an $\mathbb{R}$-divisor of  degree at most $d$ and rank equal to $r$ on $G$.
\end{theorem}

Note that Theorem \ref{bnreal_intro} implies Conjecture \ref{bakconjreal_intro} for dense multigraphs.  As an application of Theorem \ref{bnreal_intro}, we show an analogue of the gonality conjecture \cite[Conjecture 3.10]{Bak08} for $\mathbb{R}$-divisors on dense graphs and a weak analogue in general. A graph $G$ is defined to have \emph{$\mathbb{R}$-gonality} $k$ if $k$ is the infimum over the degrees of all $\mathbb{R}$-divisors on $G$ with rank at least one. 

\begin{theorem}\label{realgon_intro}
A dense graph of genus $g$ has $\mathbb{R}$-gonality at most $\lceil (g+2)/2 \rceil$.  More generally,  any undirected, connected, multigraph of genus $g$ and stretch factor $\kappa$  has $\mathbb{R}$-gonality at most $\lceil  (2\kappa-1)g/2\kappa+1 \rceil$.
\end{theorem}

Note that the gonality of a smooth, proper algebraic curve over an algebraically closed field is at most $\lfloor (g+3)/2 \rfloor$ where $g$ is the genus of the algebraic curve and is the conjectured  upper bound for graphs of genus $g$.  The upper bound $\lceil (g+2)/2 \rceil$ in Theorem \ref{realgon_intro} is equal to  $\lfloor (g+3)/2 \rfloor$ for all $g$ Note also that since the divisor with coefficient one at every vertex has rank at least one, the gonality of any  undirected, connected, multigraph is most $n$. In particular, the gonality of a dense graph is at most $\sqrt{g}+1$.  In Corollary \ref{realgon_cor}, we show analogous upper bounds for $\mathbb{R}$-gonality sequences. 


\begin{example}\label{multikn_ex}
\rm   For a positive integer $\beta$,  by $\beta \cdot G$  we denote the uniform thickening of $G$ by the factor $\beta$. The graph $\beta \cdot K_n$ for $\beta \geq 2$ is dense (and hence, has stretch factor one).  This also implies that the complete graph $K_n$ has stretch factor equal to two. Furthermore, a simple graph $G$ has stretch factor two if and only  $G$ is a complete graph\footnote{We thank the anonymous referee for this remark. }. The stretch factor of the Kneser graph $K_{s,k}$ for $~s>2k$ and fixed $k \neq 1$ stabilises at three  for large enough $s$ (the least possible for any simple graph that is not a complete graph).   In these cases, Theorem \ref{realgon_intro} gives an upper bound on the $\mathbb{R}$-gonality that, up to a constant, agrees with the one predicted by the existence conjecture. \qed
\end{example}

We then turn to Brill-Noether existence for divisors on graphs (Conjecture \ref{bakconjv1}).  In Subsection \ref{round_subsect}, we show an approximate version of Brill-Noether existence that states the following:

\begin{theorem}\label{bnapprox_intro}
Let $G$ be an undirected, connected, multigraph with $n \geq 2$ vertices, with stretch factor $\kappa$ and genus $g$.  For any pair $(r_0,d_0)$ of integers with $2g-2 \geq d_0 \geq 0$ satisfying $\tilde{\rho}(g,r_0,d_0)=g-\kappa(r_0+1)(g-d_0+r_0) \geq 0$, there exists a divisor $D$ on $G$ with degree equal to $d_0$ and rank $r$ satisfying $|r-r_0| \leq 2n-2$.
\end{theorem}

We show the following weaker version of the existence conjecture for complete graphs in degree $g-1$.






\begin{theorem}\label{bnKn_theo}
The complete graph $K_n$ on $n \geq 3$ vertices has a divisor of degree $g-1$ and rank at least  $\sqrt{g}/4-1$.
Furthermore,  if $n$ is an odd integer then it has a divisor of degree $g-1$ and rank at least $\sqrt{g}/{\sqrt{2}}-1$.
\end{theorem}


{\bf Methods:}   The Laplacian lattice of a graph is the lattice generated by the rows of its Laplacian matrix. It is a sublattice of full rank (rank $n-1$) of the root lattice $A_{n-1}$ (\cite[Section 1]{AmiMan10}).  We build on \cite{AmiMan10},~\cite{Man13} where the Laplacian lattice was studied with respect to simplicial distance functions corresponding to $\triangle$ and $\bar{\triangle}$. We initiate the study of the Laplacian lattice of a graph with respect to polyhedral distance functions corresponding to $P_{1,\lambda}$. Our starting point is a geometric interpretation of  the rank of a divisor on a graph from the author's dissertation \cite[Theorem 5.1.13]{Man11} (Proposition \ref{geomrank_prop}).  We use this geometric interpretation to formulate the covering radius conjecture (Conjecture \ref{geomverconj_intro}).   Traditionally, Brill-Noether theory seeks to minimise the degree over divisors with a fixed rank (or equivalently, maximise the rank over divisors with a fixed degree) and studies spaces $W^r_d$ \cite[Section 3, Chapter 4]{ArbCorGriHar85} of divisors (or linear series) of degree $d$ and rank at least $r$.   The key change in perspective that leads to Conjecture \ref{geomverconj_intro} is to parameterise divisors according to the ratio $(r(K_G-D)+1)/(r(D)+1)$ where $K_G$ is the canonical divisor of $D$ and to maximise the rank of divisors over this set.

In the following, we describe the key ideas behind the proofs of our main results. The proof strategy for Theorem \ref{bnreal_intro} is as follows. We first show the case of dense multigraphs (Proposition \ref{bndense_intro}). The main tools for the proof of Proposition \ref{bndense_intro} are the following. 

\begin{enumerate}

\item We compute the covering radius of $\mathcal{N}_G$ with respect to the standard simplices $\triangle$ and $\bar{\triangle}$ (Proposition \ref{covcritsimpcru_prop}). 

\item We then derive a norm conversion inequality (Corollary \ref{covradlb_cor}) that lower bounds the covering radius of $\mathcal{N}_G$ with respect to $P_{1,\lambda}$ in terms of those with respect to $\triangle$ and $\bar{\triangle}$.
 
 \end{enumerate}
 
  We then verify that these lower bounds imply Proposition \ref{bndense_intro}.  We refer to Subsection \ref{bndenseproof_subsect} for more details.   We show Theorem \ref{bnreal_intro} by combining Proposition \ref{bndense_intro} with the following scaling argument. Given an undirected, connected, multigraph $G$, we consider the dense multigraph $\kappa \cdot G$ (as in Example \ref{multikn_ex}) and apply Proposition \ref{bndense_intro} to it. We then use the observation that $L_{\kappa \cdot G}=\kappa \cdot L_G$ (where $L_{\kappa \cdot G}$ and $L_G$ are the corresponding Laplacian lattices) to show a lower bound on the covering radius of $\mathcal{N}_G$ with respect to $P_{1, \lambda}$ (Corollary \ref{covlow_cor}) from which Theorem \ref{bnreal_intro} follows.  Theorem \ref{realgon_intro} essentially follows by applying Theorem \ref{bnreal_intro} to the pair $(r_0,d_0)=(1,\lceil (2\kappa-1)g/2\kappa+1) \rceil$ (we refer to Corollary \ref{realgon_cor} for more details).  

We prove Theorem \ref{bnapprox_intro} by rounding off an $\mathbb{R}$-divisor guaranteed by Theorem \ref{bnreal_intro}.  The proof of Theorem \ref{bnKn_theo} follows the same lines as the proof of Proposition \ref{bndense_intro}, we obtain a lower bound on the (integral) version of the covering radius of $\mathcal{N}_G$  with respect to the unit ball of the $\ell_1$-norm (see Conjecture \ref{geomverell1_conj}) and translate it into a statement about the existence of divisors of degree $g-1$ and large rank. 


Finally, we remark that our methods shed light on Brill-Noether existence mainly on graphs that are sufficiently dense. On the other hand, sparse graphs such as $d$-regular graphs for a fixed $d$ still remain out of reach.  We believe that extending our results to the general case requires a deeper understanding of the Laplacian lattice of a graph with respect to polyhedral distance functions, particularly those associated to the polytope $P_{1,\lambda}$. 


{\bf A Comparison:}  The following table is a comparison of the upper bounds on the degree of an $\mathbb{R}$-divisor of rank at least $r_0$ where $0 \leq r_0 \leq g-2$ obtained from: i. the fact that the canonical divisor has degree $2g-2$  and rank $g-1$ (referred to as Bound 1 in the table below), ii. the observation that the divisor $r_0(1,\dots,1)$ has rank at least $r_0$ (referred to as Bound 2 in the table below), iii. the existence conjectures (the $\mathbb{R}$-divisor version and the original version respectively) and iv.  Corollary \ref{realgon_cor} and Theorem \ref{bnapprox_intro} (along with Riemann-Roch) respectively, we refer to them as ``our bounds''.  We refer to these two parameters as the $r_0$-th $\mathbb{R}$-gonality and $r_0$-th gonality respectively.

\begin{center}
 \begin{tabular}{|| c| c | c | c | c ||} 
 \hline
  Parameter &  Bound 1 & Bound 2 & Existence Conjectures & Our Bounds\\ [0.5ex] 
 \hline
 $r_0$-th $\mathbb{R}$-Gonality & $g-1+r_0$  & $n \cdot r_0$ &$g-g/(r_0+1)+r_0$&  $c_1g+r_0$ \\
 \hline
 $r_0$-th Gonality &  $g-1+r_0$ & $n \cdot r_0$  & $\lceil g-g/(r_0+1)+r_0 \rceil$ &   $\lceil c_2 g+c_3 \rceil$   \\
 \hline
 \end{tabular}
 
 \end{center}
 
 where  $c_1=1-1/(\kappa(r_0+1)), c_2=1-1/(\kappa(r_0+2n-1))$ and $c_3=r_0+2n-2$.  
 
 In the case of $\mathbb{R}$-divisors, our bound improves upon both bound 1 and bound 2 when $r_0 \cdot (n-1)/c_1 \geq g \geq \kappa(r_0+1)$. Asymptotically in $n$, this happens, for instance, if $g \in \omega(n^{3/2})$ but is subquadratic in $n$  (and hence, $\kappa \in o(n^{1/2})$)   and $r_0$ is linear in $n$.  In the case of divisors, our bound improves upon bound 1 when $r_0 \leq g/(2n-1)\kappa-2n$  (this happens, for instance, if $G$ is a graph satisfying $g \geq  3n(2n-1)$ then $G$ is dense, i.e. $\kappa=1$ and the improvement holds for all $r_0 \leq n$).  Our bound improves upon bound 2 if $g \in o(n^2)$ (and hence, $\kappa \in \omega(1)$) and $r_0 \in \Omega(n)$.  In the case of divisors, a simple calculation shows that our bound does not substantially, i.e. in the order of growth with respect to $n$, improve upon both bound 1 and bound 2. We still include Theorem \ref{bnapprox_intro} in anticipation of further refinements along these lines.

Throughout the rest of the paper, we identify $\mathbb{R}$-divisors on $G$ with points in $\mathbb{R}^n$ by identifying the standard basis of $\mathbb{R}^n$ with the vertices of $G$ via some fixed bijection.

{\bf Acknowledgement:} We thank Omid Amini, Spencer Backman, Matt Baker, Ye Luo, Dhruv Ranganathan and Chi Ho Yuen for illuminating discussions on Brill-Noether theory of both curves and graphs.   We thank the anonymous referee very much for the several very valuable suggestions on the manuscript, we have benefited very much from them.



\section{$\mathbb{R}$-Divisor Theory on Graphs}

Let $G=(V,E)$ be an undirected connected graph with $n$ vertices,~ $m$ edges and genus $g~(=m-n+1)$.  Let $A$ be an Abelian group.  Let ${\rm Div}_{A}(G) \cong A^n$ be the Abelian group whose elements are formal sums of the form $\sum_{v \in V} a_v(v)$ where $a_v \in A$ and with the group structure given by $\sum_v a_v(v)+\sum_v{b_v}(v)=\sum_v (a_v+b_v)(v)$. Usually in the divisor theory of both algebraic varieties and graphs, $A$ is taken to be the integers. We refer to this as \emph{standard divisor theory}. We consider the case where $A=\mathbb{R}$ and  refer to elements in ${\rm Div}_{\mathbb{R}}(G)$ as \emph{$\mathbb{R}$-divisors}.  We recall analogues of degree of a divisor, effective divisors, linear systems of divisors and rank (as sketched in \cite[Section 8.1]{AmiMan10} and  developed by James and Miranda \cite{JamMir13},~\cite{JamMir14}).  These analogues, except for rank of an $\mathbb{R}$-divisor, are straightforward generalisations of the corresponding notions in standard divisor theory.

\begin{definition}

The degree ${\rm deg}:{\rm Div}_A(G) \rightarrow A$ of a divisor is the homomorphism  given by $\sum_v a_v(v) \rightarrow \sum_v a_v$.

\end{definition}

\begin{definition}{\rm ({\bf Effective Divisor}) }
An  $\mathbb{R}$-divisor $D=\sum_v a_v(v) \in {\rm Div}_{\mathbb{R}}(G)$ is called effective if $a_v \geq 0$ for all $v$.  
\end{definition}

Rational functions on $G$ are exactly as in standard divisor theory \cite{BakNor07}.  A rational function on $G$ is a function $f:V \rightarrow \mathbb{Z}$. The set of rational functions are also naturally equipped with a group structure.  We denote this group by $\mathcal{M}(G)$. Let $Q: \mathcal{M}(G) \rightarrow \rm{Div}_{\mathbb{R}}(G)$ be the Laplacian operator on $G$ defined as $Q(f)=\sum_{v \in V}(\sum_{(u,v) \in E}(f(v)-f(u)))(v)$.  Linear equivalence is exactly as in the case of standard divisor theory. 

\begin{definition}{\rm({\bf Linear Equivalence})}
Let $D_1$ and $D_2$ be $\mathbb{R}$-divisors on $G$, then $D_1 \sim D_2$ if $D_1-D_2=Q(f)$ for some $f \in \mathcal{M}(G)$. 
\end{definition}

The linear system $|D|$ of an  $\mathbb{R}$-divisor $D$ is defined to be the set of all effective  $\mathbb{R}$-divisors linearly equivalent to it. 

\begin{definition}{\rm ({\bf Rank of an $\mathbb{R}$-Divisor})}
The modified rank $\tilde{r}(D)$ of an $\mathbb{R}$-divisor is defined as $k_0-1$ where $k_0$ is the infimum over all $k$ such that there exists an effective $\mathbb{R}$-divisor $E$ of degree $k$ such that $|D-E|=\emptyset$. The rank $r(D)$ is defined as $\tilde{r}(D+\sum_{v \in V}(v))$. 
\end{definition}

If $D$ is a divisor, then its rank $r(D)$ as an $\mathbb{R}$-divisor agrees with its Baker-Norine rank, see Appendix \ref{rrreal_sect}.  James and Miranda \cite{JamMir13} also proved a Riemann-Roch theorem for $\mathbb{R}$-divisors on graphs (more generally, on real edge weighted graphs). For easy reference, we include a Riemann-Roch theorem for $\mathbb{R}$-divisors on graphs in Appendix \ref{rrreal_sect}. This is essentially a special case of the Riemann-Roch theorem of James and Miranda.

The following $\mathbb{R}$-divisor version of the degree plus formula for rank, due to Baker and Norine in the case of divisors \cite[Lemma 2.7]{BakNor07}, plays a key role.

\begin{proposition}\label{degplusrank_theo}
The rank of an  $\mathbb{R}$-divisor $D_{\mathbb{R}}$ is equal to 

\begin{center} ${\rm min}_{\nu \in \mathcal{N}_G} {\rm deg}^{+}(D_{\mathbb{R}}-\nu)-1$, \end{center} 

where for an  $\mathbb{R}$-divisor $D'_{\mathbb{R}}=\sum_v a_v(v)$, the degree-plus function ${\rm deg}^{+}(D'_{\mathbb{R}})$ at $D'_{\mathbb{R}}$ is $\sum_{a_v>0}a_v$. 
\end{proposition}

\begin{proof}
Follows from Proposition \ref{rankrealdiv_prop} and the fact that $\widetilde{{\rm Ext}}^c(G)$ is equal to $\mathcal{N}_G+\sum_v (v)$ (see Appendix \ref{rrreal_sect} for more details).
\end{proof}


\subsection{Brill-Noether Theory of $\mathbb{R}$-Divisors on Graphs}\label{bnreal_subsect}

In this subsection, we formulate Brill-Noether existence and non-existence for $\mathbb{R}$-divisors on graphs.

\begin{conjecture}\label{bakconjreal} {\rm{({\bf Brill-Noether Existence for $\mathbb{R}$-Divisors on $G$})}}   Fix two real numbers $r \geq 0,~ 2g-2 \geq d \geq 0$ and set $\rho(g,r,d)=g-(r+1)(g-d+r)$.
 If  $\rho(g,r,d) \geq 0$, then there exists an $\mathbb{R}$-divisor of degree at most $d$ and rank equal to $r$ on $G$. Furthermore, if $g-d+r \geq 1$ then there exists an $\mathbb{R}$-divisor of degree equal to $d$ and rank equal to $r$ on $G$.
\end{conjecture}

\begin{remark}\label{realhighdeg_rem}
\rm Unlike in the case of divisors on graphs, Brill-Noether existence for $\mathbb{R}$-divisors does not directly generalise for $d>2g-2$.   As an example, suppose that $G$ has genus at least one and consider $d_0=2g-1+n$ and $r_0=d_0-g+\epsilon$ for a positive real number $\epsilon$. Note that $\rho(g,r_0,d_0) \geq 0$ if $\epsilon$ is sufficiently small. However, since $d_0=2g-1+n>2g-2+n$,  by Remark \ref{ranksub_rem}, Item \ref{third_item},  any $\mathbb{R}$-divisor $D$ of degree $d_0$ satisfies $r(D)=d_0-g<d_0-g+\epsilon=r_0$. Hence,  $G$ does not have a divisor of degree at most $d_0$ and rank equal to $r_0$. This subtlety arises since we allow $r,~d$ to be real numbers and can be resolved by restricting them to the integers. Furthermore, for $\mathbb{R}$-divisors of degree strictly greater than $2g-2+n$ the rank is equal to the degree minus $g$ and addresses the Brill-Noether theory in this range.   We will not deal with the range $(2g-2,2g-2+n]$  in the present work and point to Remark  \ref{ranksub_rem}, Item \ref{third_item} for hints in this direction.  \qed
\end{remark}

\begin{remark}\label{redreal_rem} \rm
We note that it suffices to prove Conjecture \ref{bakconjreal} for $0 \leq d \leq g-1$.  More generally, we note that this holds even if $\rho(g,r,d)$ in Conjecture \ref{bakconjreal} is replaced by $\tilde{\rho}(g,r,d)=g-\kappa'(r+1)(g-d+r)$ for any real number $\kappa'$. 
Given a pair $(r_0,d_0)$ such that $g-1<d_0 \leq 2g-2$ and $\rho(g,r_0,d_0) \geq 0$, then let $d'_0=2g-2-d_0$ and $r'_0=r_0-d_0+(g-1)$. By construction, $\rho(g,r_0,d_0)=\rho(g,r'_0,d'_0)$ and $0 \leq d'_0 < g-1$.

If $r'_0 \geq 0$, Conjecture \ref{bakconjreal} gives a divisor $D$ of degree $d'_0$ and rank $r'_0$ on $G$.  By the Riemann-Roch theorem for $\mathbb{R}$-divisors on graphs (Theorem \ref{rrreal_theover2}), the divisor $K_G-D$ where $K_G$ is the canonical divisor of $G$ has degree $2g-2-d'_0=d_0$ and rank $r(D)-d'_0+(g-1)=r'_0-d'_0+(g-1)=r_0$.  On the other hand, if $r'_0<0$ then we consider any  effective divisor $D$ of degree $d'_0$. Note that $D$ has non-negative rank. The divisor $K_G-D$  has degree $d_0$ and an application of the Riemann-Roch theorem shows that its rank is at least $r_0$.   The continuity of the rank function (Corollary \ref{rankcont_cor}) implies that there exists a divisor of degree at most $d_0$ and rank $r_0$.  \qed




\end{remark}


On the non-existence side, we conjecture the following.

\begin{conjecture}\label{realnonex_conj} {\rm{({\bf Brill-Noether Non-existence for $\mathbb{R}$-Divisors})}} There exists a family of undirected, connected, multigraphs $\{G_g\}$ one for each genus $g$ such that for any two real numbers $r \geq 0,~2g-2 \geq d \geq 0$ if $\rho(g,r,d)< 0$ then there does not exist an $\mathbb{R}$-divisor of degree at most $d$ and rank equal to $r$ on $G_g$.   \end{conjecture}


\section{Covering Radius of Discrete Sets with a Lattice Action}\label{covrad_sect}

In this section, we set up the geometric framework for the covering radius conjecture. In particular, we define the notion of covering radius of a discrete set carrying the action of a lattice and study it in cases that are relevant to Brill-Noether theory on graphs. 

For a real number $d$, let $H_{d}=\{ (x_1,\dots,x_n) \in \mathbb{R}^n|~\sum_i x_i=d\}$.  Let $C$ be a convex body (a compact, convex set) in $H_0$ containing the origin (referred to as the centre of $C$). The convex body $C$ induces a distance $d_C: H_0 \times H_0 \rightarrow \mathbb{R}_{\geq 0}$ (see \cite[Chapter 4, Page 103]{Cas71}) defined as:
 
 \begin{center}$d_C({\bf p_1},{\bf p_2})={\rm min}\{\mu|~{\bf p_2} \in {\bf p_1}+\mu \cdot C\}$. \end{center}
 
In general, the function $d_C$ satisfies all properties of a metric except symmetry: $d_C({\bf p_1},{\bf p_2})$ need not be equal to  $d_C({\bf p_2},{\bf p_1})$. However, if $C$ is centrally symmetric, i.e. $-{\bf x} \in C$ for all ${\bf x} \in C$ then $d_C({\bf p_1},{\bf p_2})$ is indeed a metric.  Some instances of $C$ that arise in our context are the standard simplices $\triangle,\bar{\triangle}$ in $H_0$ defined as:
          \begin{center}   $\triangle={\rm CH} ({\bf b_1},\dots, {\bf b_n})$ where ${\bf b_i}=(n-1) \cdot {\bf e_i}-\sum_{j \neq i}{\bf e_j}$, \end{center}
 
 where ${\rm CH}(.)$ is the convex hull and ${\bf e_1},\dots,{\bf e_n}$ are the standard basis vectors of $\mathbb{R}^n$. The standard simplex $\bar{{\bf \triangle}}$ is defined as $-\triangle$. Note that the origin is contained in both $\triangle$ and $\bar{\triangle}$.  As in \cite{AmiMan10}, we refer to the corresponding distance functions $d_{\triangle}$ and $d_{\bar{\triangle}}$ as the \emph{simplicial distance functions}.  More generally, we will deal with Minkowski sums of dilates of these two standard simplices, i.e. $\alpha \cdot \triangle+\bar{\alpha} \cdot \bar{\triangle}$ where $\alpha,~\bar{\alpha} \in \mathbb{R}_{\geq 0}$.
 
 \begin{remark}\label{simpdisform_rem}
 \rm By \cite[Lemma 4.7]{AmiMan10}, the simplicial distance functions $d_{\triangle}$ and $d_{\bar{\triangle}}$ are given by the following formulas:
 
 \begin{center}
    $d_{\triangle}({\bf p},{\bf q})=|{\rm min}_i (q_i-p_i)|$,\\ 
    $d_{\bar{\triangle}}({\bf p},{\bf q})=|{\rm max}_i (q_i-p_i)|$
  \end{center}  
    
     for any ${\bf p}=(p_1,\dots,p_n)$ and ${\bf q}=(q_1,\dots,q_n)$ in $H_0$.

 \end{remark}

  Let $\Lambda$ be a lattice of full rank in $H_0$. Note that the lattice $\Lambda$ acts on $H_d$, for any $d \in \mathbb{R}$, by translation.  Let $T$ be a subset of $H_d$, for some real number $d$, that is discrete with respect to the Euclidean topology and inheriting the action of $\Lambda$ on $H_d$ with finitely many orbits. We equip $H_d$ with a distance function $h_{C,T}: H_d \rightarrow \mathbb{R}$ with respect to $C$ and $T$ (this generalises the function $h_{C,\Lambda}$ as defined in \cite[Chapter 4.4, Page 119]{Cas71}) as follows.

\begin{center}$h_{C,T}({\bf p})={\rm min}_{\mu \in T}d_C({\bf p},\mu)$. \end{center}

\begin{remark}\label{hct_rem}
\rm
Note that $h_{C,T}$ is a $\Lambda$-periodic, continuous function on $H_d$. Hence, $h_{C,T}$ induces a continuous function on the (topological) torus $H_{d}/{\sim_{\Lambda}}$. Thus, it is bounded and attains its infimum and supremum. Its minimum is zero and is attained at points in $T$.  We define its maximum to be the covering radius of $T$ with respect to $C$.   \qed
\end{remark}

\begin{definition}{({\bf Covering Radius})}
The maximum value of the function $h_{C,T}: H_{d} \rightarrow \mathbb{R}$ is called the covering radius of $T$ with respect to $C$ and is denoted by ${\rm Cov}_{C}(T)$. 
\end{definition}

\begin{remark}
\rm The covering radius can also be defined as the maximum circumradius (with respect to $C$) over all the Voronoi cells of $T$ with respect to the distance function induced by $C$. \qed
\end{remark}
 \begin{remark}
\rm  Covering radii of lattices (with respect to the Euclidean and polyhedral distance functions) have been widely studied in various contexts, for instance the geometry of numbers and coding theory \cite{ConSlo13}.  \qed
 \end{remark}

 The following instances play a key role in the context of Brill-Noether theory of graphs.

\begin{enumerate}
\item  $\Lambda=L_G$, the Laplacian lattice of an undirected connected multigraph $G$ of genus $g$, $T=\mathcal{N}_G \subseteq H_{g-1}$, where $\mathcal{N}_G$ is the set of divisors of degree $g-1$ and rank $-1$. By definition, this set carries the action of $L_G$. The number of orbits of this action  is precisely the number of acyclic orientations on $G$ with a unique sink at a fixed vertex $v$. Each orbit corresponds to a divisor class and has a unique $v$-reduced representative \cite[Section 3.1]{BakNor07}.  
 This unique $v$-reduced representative is the divisor whose coefficient at any vertex is its outdegree (with respect to the acyclic orientation or equivalently, a permutation of the vertices inducing it) minus one.  The convex body $C$ is of the form $\triangle+\lambda{\bar{\triangle}}$ for some non-negative real number $\lambda$.

\item $\Lambda=T=L_G \subset H_0$ and $C=\triangle+\lambda{\bar{\triangle}}$ for some $\lambda \in \mathbb{R}_{\geq 0}$. The case where $C=\triangle$ and $C=\bar{\triangle}$ has been studied in \cite{AmiMan10}.

\item  $\Lambda=L_G$ and $T={\rm Crit}_{\triangle}(L_G)$ or $T={\rm Crit}_{\bar{\triangle}}(L_G)$ where ${\rm Crit}_{\triangle}(L_G)$ and ${\rm Crit}_{\bar{\triangle}}(L_G)$ are the set of local maxima of the functions $h_{\triangle, L_G}$ and $h_{\bar{\triangle},L_G}$ respectively \cite{AmiMan10}.    For a point ${\bf p}=(p_1,\dots,p_n) \in \mathbb{R}^n$,  let $\pi_0({\bf p}):={\bf p}-\sum_i p_i/n \cdot (1,\dots,1)$ be the orthogonal projection of ${\bf p}$ onto the hyperplane $H_0$.  In the following example and the upcoming subsection, we extensively use the following version of \cite[Lemma 4.11]{AmiMan10}.

\begin{proposition}\cite[Lemma 4.11]{AmiMan10}\label{critchar_lem}
The set ${\rm Crit}_{\triangle}(L_G)$ is equal to  $-\pi_0(\mathcal{N}_G)$.
\end{proposition}

Since every divisor $D$ in $\mathcal{N}_G$ has rank minus one, we know that every  divisor $D'$ that is linearly equivalent to $D$ has negative coefficient at some vertex. Hence, $|{\min}_i(\pi_0(D'))_i|$ is at least $m/n$. Furthermore, for any vertex $v$ and a divisor class in $\mathcal{N}_G$,  its $v$-reduced representative $D_{\mathcal{O}}$ satisfies $|{\min}_i(\pi_0(D_{\mathcal{O}}))_i|=m/n$. Hence, by Remark \ref{simpdisform_rem} and Proposition \ref{critchar_lem},  it follows that the value of $h_{\triangle, L_G}$ at every point in ${\rm Crit}_{\triangle}(L_G)$ is $m/n$ \footnote {This is essentially the notion of uniformity \cite[Definition 2.12] {AmiMan10}.}.  Hence, the covering radius of $L_G$ with respect to $\triangle$ is $m/n$.

 Another important property of $L_G$ is \emph{reflection invariance}, i.e. there is an element $t \in H_0$ such that  ${\rm Crit}_{\triangle}(L_G)=-{\rm Crit}_{\triangle}(L_G)+t$, we refer to \cite[Theorem 6.1]{AmiMan10} for a stronger version of this statement. Furthermore, one choice of $t$ is $\pi_0(K_G)$ where $K_G$ is the canonical divisor of $G$. Reflection invariance is a consequence of Proposition \ref{critchar_lem} and the observation that the canonical divisor $K_G$ can be expressed as the sum of two reduced divisors in $\mathcal{N}_G$: one associated to a permutation and the other associated to its opposite permutation. Reflection invariance combined with the observation that ${\rm Crit}_{\bar{\triangle}}(L_G)=-{\rm Crit}_{\triangle}(L_G)$ implies that  the covering radius of $L_G$ with respect to $\bar{\triangle}$ is also $m/n$.  Note that \cite{AmiMan10} uses the terminology ${\rm Crit}(L_G)$ for ${\rm Crit}_{\triangle}(L_G)$.

\end{enumerate}

\begin{example}
\rm For the complete graph $K_n$, the Laplacian lattice $L_{K_n}$ is a sublattice of $A_{n-1}$ generated by $(n-1,-1,\dots,-1),(-1,n-1,-1,-1,\dots,-1),\\\dots,(-1,-1,\dots,-1,n-1)$ and has index $[A_{n-1}:L_G]=n^{n-2}$.  The set $\mathcal{N}_G \subset H_{g-1}$ of non-special divisors is $\{\nu_{\sigma}+L_G\}_{\sigma \in S_n}$ where $\nu_\sigma=\sigma((n-2,n-3,n-4,\dots,-1))$ and each element of $\mathcal{N}_G$ has degree $\binom{n-1}{2}-1$. Using Proposition \ref{critchar_lem}. the set ${\rm Crit}_{\triangle}(L_G) \subset H_0$ is $\{c_{\sigma}+L_G\}_{\sigma \in S_n}$ where $c_\sigma=\sigma(((n-1)/2,(n-3)/2,(n-5)/2,\dots,-(n-1)/2))$.  Furthermore, in this case ${\rm Crit}_{\triangle}(L_G)={\rm Crit}_{\bar{\triangle}}(L_G)$. The covering radii of $L_G$, ${\rm Crit}_{\triangle}(L_G)$ and ${\rm Crit}_{\bar{\triangle}}(L_G)$ are all equal to $(n-1)/2$ (see  Proposition \ref{critchar_lem}, the paragraph following it and Proposition \ref{covcritsimpcru_prop}). 

Consider the Petersen graph (the Kneser graph $K_{5,2}$),  the standard basis vectors of $\mathbb{Z}^{10}$ are labeled by its vertices, i.e. by subsets of size two in $[1,\dots,5]$ and hence, we denote a standard basis vector by $e_{i,j}$ for $i \neq j$, $i,~j \in [1,\dots,5]$.  
The Laplacian lattice of the Petersen graph $K_{5,2}$ is a sublattice generated by the vectors $3 \cdot e_{\sigma(1),\sigma(2)}-e_{\sigma(3),\sigma(4)}-e_{\sigma(4),\sigma(5)}-e_{\sigma(3),\sigma(5)}$ over all permutations of $[1,\dots,5]$. Its index (as a sublattice of $A_{9}$) is 2000 \cite[Page 285]{HolShe93}. It has 16,680 acyclic orientations and the number of orbits of the action of its automorphism group on the set of acyclic orientations is 168 \cite{CamGla12}. The covering radii of $L_{K_{5,2}}$, ${\rm Crit}_{\triangle}(L_{K_{5,2}})$ and ${\rm Crit}_{\bar{\triangle}}(L_{K_{5,2}})$ are all $3/2$.  \qed

\end{example}


\subsection{Covering radius of  ${\rm Crit}_{\triangle}(L_G)$ with respect to $\triangle$ and $\bar{\triangle}$}

In the following, we compute the covering radius of ${\rm Crit}_{\triangle}(L_G),~{\rm Crit}_{\bar{\triangle}}(L_G)$ and $\mathcal{N}_G$ with respect to the distance induced by the standard simplices $\triangle$ and $\bar{\triangle}$.

\begin{proposition}\label{covcritsimpcru_prop} The covering radii of ${\rm Crit}_{\triangle}(L_G)$ with respect to  $\triangle$ and $\bar{\triangle}$ are both equal to $m/n$, where $m$ is the number of edges and $n$ is the number of vertices  of $G$.   \end{proposition}

\begin{proof}
We start by noting that the covering radii of $L_G$ with respect to $\triangle$ and $\bar{\triangle}$ are both equal to $m/n$.  This is a consequence of Proposition \ref{critchar_lem} and the fact that $L_G$ is reflection invariant.  
 Combining this with \cite[Theorem 8.3]{AmiMan10} (this is essentially a consequence of Proposition \ref{sigmachar_prop})  yields ${\rm Cov}_{\bar{\triangle}}({\rm Crit}_{\triangle}(L_G))={\rm Cov}_{\triangle}({\rm Crit}_{\bar{\triangle}}(L_G))=m/n$. Next, we note that since $L_G$ is a lattice and $\bar{\triangle}=-\triangle$, we have ${\rm Crit}_{\bar{\triangle}}(L_G)=-{\rm Crit}_{\triangle}(L_G)$.  Since the Laplacian lattice is reflection invariant, we have ${\rm Crit}_{\triangle}(L_G)=-{\rm Crit}_{\triangle}(L_G)+t={\rm Crit}_{\bar{\triangle}}(L_G)+t$ for some $t \in H_0$. Since the covering radius is preserved by translation, we conclude that ${\rm Cov}_{\triangle}({\rm Crit}_{\triangle}(L_G))=m/n$.
  \end{proof}

  \begin{remark}
  \rm The fact that the covering radius of $L_G$ with respect to $\triangle$ is equal to $m/n$ also follows from the Riemann-Roch theorem for $\mathbb{R}$-divisors and the fact that elements in $\mathcal{N}_G$ have rank minus one.  Since $d_{\triangle}({\bf p},{\bf q})={\rm max}_i (p_i-q_i)$,  it is also equivalent to every $\mathbb{R}$-divisor of degree zero being linearly equivalent to an $\mathbb{R}$-divisor where every coefficient is at most $m/n$ \footnote{We thank the anonymous referee for this remark.}. We shall omit the details here.  \qed

  \end{remark}

  

Recall that for a point ${\bf p}=(p_1,\dots,p_n) \in \mathbb{R}^n$,  let  $\pi_0({\bf p}):={\bf p}-\sum_i p_i/n \cdot (1,\dots,1)$. We refine Proposition \ref{covcritsimpcru_prop} to describe a coset of $L_G$ where the maximum of $h_{\triangle,{\rm Crit}_{\triangle}(L_G)}$ is attained. Along the way, we give another proof of the fact that ${\rm Cov}_{\triangle}({\rm Crit}_{\triangle}(L_G))=m/n$.  Recall that by Proposition \ref{critchar_lem},   ${\rm Crit}_{\triangle}(L_G)$ is equal to $-\pi_0(\mathcal{N}_G)$ and hence,  ${\rm Crit}_{\triangle}(L_G)$ is a finite union of orbits under the action of $L_G$ on $H_0$. Hence, ${\rm Crit}_{\triangle}(L_G)$ is a finite union of translates of $L_G$.

\begin{proposition}\label{covcritsimp_prop} The covering radius of ${\rm Crit}_{\triangle}(L_G)$ with respect to  $\triangle$ is equal to $m/n$ and the function  $h_{\triangle,{\rm Crit}_{\triangle}(L_G)}$ attains a maximum at $\pi_0(K_G)+L_G$ where $K_G=\sum_{v \in V} ({\rm val}(v)-2)(v)$ (where ${\rm val}(v)$ is the valency of $v$) is the canonical divisor of $G$. \end{proposition}

\begin{proof} Note that ${\rm Crit}_{\triangle}(L_G)$ is a  finite union of translates of $L_G$ and that the covering radius is preserved by translation. Furthermore, the covering radius of any translate  $c+L_G$ where $c \in {\rm Crit}_{\triangle}(L_G)$ is at least the covering radius of ${\rm Crit}_{\triangle}(L_G)$.  Hence, the covering radius of ${\rm Crit}_{\triangle}(L_G)$ with respect to $\triangle$ is upper bounded by the corresponding covering radius of $L_G$ which in turn is equal to $m/n$ (see the paragraph following Proposition \ref{critchar_lem}).  Hence, ${\rm Cov}_{\triangle}({\rm Crit}_{\triangle}(L_G)) \leq m/n$.  

Next, we show that ${\rm Cov}_{\triangle}({\rm Crit}_{\triangle}(L_G)) \geq m/n$ by constructing a point ${\bf p} \in H_0$ such that $h_{\triangle,{\rm Crit}_{\triangle}(L_G)}({\bf p}) \geq m/n$. For this, consider ${\bf p}=\pi_0(K_G)$ where $K_G$ is the canonical divisor of $G$. We claim that $h_{\triangle,{\rm Crit}_{\triangle}(L_G)}({\bf p}) \geq m/n$. Note that since  ${\rm Crit}_{\triangle}(L_G)$ is reflection invariant with $t=\pi_0(K_G)={\bf p}$ (cf. the paragraph following Proposition \ref{critchar_lem}), we have $\bar{c}:={\bf p}-c \in {\rm Crit}_{\triangle}(L_G)$ for all $c \in  {\rm Crit}_{\triangle}(L_G)$. Since ${\rm Crit}_{\triangle}(L_G)$ is the set of local maxima of $h_{\triangle,L_G}$ and  the value of $h_{\triangle,L_G}$ is equal to $m/n$ at all points in ${\rm Crit}_{\triangle}(L_G)$. Thus, we have $d_{\triangle}(\bar{c},O) \geq m/n$ where $O$ is the origin. Hence, $d_{\triangle}({\bf p},c)=d_{\triangle}(\bar{c},O) \geq m/n$ for all  $c \in  {\rm Crit}_{\triangle}(L_G)$.
\end{proof}

Finally, we note that since ${\rm Crit}_{\triangle}(L_G)=\pi_0(\mathcal{N}_G)$ (Proposition \ref{critchar_lem}) and that ${\rm Crit}_{\bar{\triangle}}(L_G)={\rm Crit}_{\triangle}(L_G)-\pi_0(K_G)$, their covering radii are equal with respect to any fixed convex body. 
 Let $\pi_{g-1}(K_G)=K_G-(g-1)/n \cdot (1,\dots,1)$ be the orthogonal projection of $K_G$ onto $H_{g-1}$ (here, $K_G$ is regarded as a point in $\mathbb{R}^n$). 

\begin{proposition}\label{covradeq_prop}  For any convex body $C$, the covering radius of the three sets ${\rm Crit}_{\triangle}(L_G),~{\rm Crit}_{\bar{\triangle}}(L_G)$ and $\mathcal{N}_G$ with respect to $C$ are equal.  For $C=\triangle$, the functions   $h_{\triangle,{\rm Crit}_{\triangle}(L_G)}$,  $h_{\triangle,{\rm Crit}_{\bar{\triangle}}(L_G)}$ and $h_{\triangle,\mathcal{N}_G}$, attain their maximum over the cosets $\pi_{0}(K_G)+L_G$, $L_G$ and $\pi_{g-1}(K_G)+L_G$ respectively. 
\end{proposition}





\section{The Existence Conjecture for $\mathbb{R}$-Divisors in terms of Covering Radius}\label{bnreal_sect}


For non-negative real numbers $\alpha,~\bar{\alpha}$, define the polytope $P_{\alpha,\bar{\alpha}}$ to be the Minkowski sum of $\alpha \cdot \triangle$ and $\bar{\alpha} \cdot \bar{\triangle}$.  We denote $P_{\alpha,\bar{\alpha}}+{\bf p}$, for some point ${\bf p} \in \mathbb{R }^n$,  by $P_{\alpha,\bar{\alpha}}({\bf p})$. For an $\mathbb{R}$-divisor on $G$, we define $\pi_k(D)=D-({k-{\rm deg}(D))/n \cdot (1,\dots,1)}$ (here $D$ is regarded is a point in $\mathbb{R}^n$).

\begin{proposition}\label{geomrank_prop} Let $D$ be an $\mathbb{R}$-divisor of degree $d$.  If $d \leq g-1$, then $r(D)$ is equal to $k-1$ where $k$ is the minimum non-negative real such that $\pi_{g-1}(D) \in \cup_{\nu \in \mathcal{N}_G}P_{k/n,(g-1-d+k)/n}(\nu)$. If $d \geq g-1$, then its rank is equal to $d-g+k$ where $k$ is the minimum non-negative real such that $\pi_{g-1}(D) \in   \cup_{\nu \in \mathcal{N}_G} P_{(k+d-(g-1))/n, k/n}(\nu)$.
\end{proposition}



\begin{proof}

We start by observing that for any ${\bf p}=(p_1,\dots,p_n), {\bf q}=(q_1,\dots,q_n)$ in $\mathbb{R}^n$, the degree plus function ${\rm deg}^{+}({\bf p}-{\bf q}):=\sum_{i: p_i>q_i} (p_i-q_i)$ at ${\bf p}-{\bf q}$ can be interpreted as follows.

For ${\bf r} \in \mathbb{R}^n$, let $H^{+}({\bf r})=\{ {\bf r'} \in \mathbb{R}^n|~{\bf r'} \geq {\bf r}\}$ where ${\bf r'} \geq {\bf r}$ is coordinatewise domination. By the definition of ${\rm deg}^{+}$, we have:
\begin{center}
${\rm deg}^{+}({\bf p}-{\bf q})={\rm min}\{ {\rm deg}({\bf r})\}_{{\bf r} \in H^{+}({\bf p}) \cap H^{+}({\bf q})}-{\rm deg}({\bf q})$.
\end{center} 

 For ${\bf t} \in \mathbb{R}^n$, we use the shorthand $d_t$ for ${\rm deg}({\bf t})$. Suppose that $d_p \geq d_q$. Remark \ref{simpdisform_rem} along with a simple calculation shows that  ${\bf r} \in H^{+}({\bf p}) \cap H^{+}({\bf q})$ is equivalent to $d_{\triangle}(\pi_{d_q}({\bf p}),\pi_{d_q}({\bf r})) \leq (d_r-d_p)/n$ and $d_{\triangle}({\bf q},\pi_{d_q}({\bf r})) \leq (d_r-d_q)/n$. By the definition of $d_{\triangle}$, this in turn is equivalent to $\pi_{d_q}({\bf r}) \in ((d_r-d_p)/n) \cdot \triangle +\pi_{d_q}({\bf p})$ and $\pi_{d_q}({\bf r}) \in ((d_r-d_q)/n) \cdot \triangle+{\bf q}$.  Hence,  ${\rm min}\{ {\rm deg}({\bf r})\}_{{{\bf r}} \in H^{+}({\bf p}) \cap H^{+}({\bf q})}$ is equal to $k+d_p$ where $k$ is the minimum non-negative real number such that $(k/n \cdot \triangle+\pi_{d_q}({\bf p})) \cap ((k+d_p-d_q)/n \cdot \triangle+{\bf q}) \neq \emptyset$.  This $k$ is in turn equal to the minimum non-negative real $\tilde{k}$ such that $\pi_{d_q}({\bf p}) \in ((\tilde{k}+d_p-d_q)/n \cdot {\triangle}+\tilde{k}/n\cdot \bar{\triangle})+{\bf q}$. 
On the other hand, suppose that $d_p \leq d_q$ then ${\rm min}\{ {\rm deg}({\bf r})\}_{{\bf r} \in H^{+}({\bf p}) \cap H^{+}({\bf q})}$  is equal to $\tilde{k}+d_q$ where $\tilde{k}$ is the minimum non-negative real such that $\pi_{d_q}({\bf p}) \in (\tilde{k}/n \cdot \triangle+(\tilde{k}+d_q-d_p)/n \cdot \bar{\triangle})+{\bf q}$.

We now specialise to ${\bf p}=D$ for some $\mathbb{R}$-divisor $D$ with degree $d$ and ${\bf q}=\nu$ for some $\nu \in \mathcal{N}_G$. Hence, $d_p=d$ and $d_q=g-1$. Using the degree-plus formula for rank of an $\mathbb{R}$-divisor (Proposition \ref{degplusrank_theo}), 
we obtain the following.

If $d \geq g-1$, then  $r(D)$ is equal to  $d-g+k'$ where $k'$ is the minimum non-negative real number such that 

\begin{center}
$\pi_{g-1}(D) \in  \cup_{\nu \in \mathcal{N}_G}(((k'+d-(g-1))/n \cdot \triangle+k'/n \cdot \bar{\triangle})+\nu)$.
\end{center}

If $d \leq g-1$, then $r(D)$ is equal to the minimum non-negative real number $k'-1$ such that 

\begin{center}
$\pi_{g-1}(D) \in  \cup_{\nu \in \mathcal{N}_G} ((k'/n \cdot \triangle+ (g-1-d+k')/n \cdot \bar{\triangle})+\nu)$.
\end{center}

This concludes the proof of Proposition \ref{geomrank_prop}.  
\end{proof}

%




In \cite[Theorem 5.1.13]{Man11}, Proposition \ref{geomrank_prop} is stated only for a divisor $D$ with $0 \leq {\rm deg}(D) \leq g-1$.    Note that if $r(D)+1=k$ then $r(K_G-D)+1=g-1-{\rm deg}(D)+k$ and if $r(K_G-D)+1=k$ then $r(D)+1={\rm deg}(D)-(g-1)+k$.  These pairs are precisely the parameters $\alpha,~\bar{\alpha}$ for $P_{\alpha,\bar{\alpha}}$ in Proposition \ref{geomrank_prop}. This leads us to the following reformulation of the existence conjecture for $\mathbb{R}$-divisors (Conjecture \ref{bakconjreal}).  

\begin{conjecture}\label{geomver_conj} {\rm(\bf{Covering Radius Conjecture})} Let $\lambda \in [1/g,g]$ (recall that $g \geq 1$). The covering radii of $\mathcal{N}_G, {\rm Crit}_{\triangle}(L_G)$ and ${\rm Crit}_{\bar{\triangle}}(L_G)$ with respect to the distance function induced by $P_{1,\lambda}$ is at least $\sqrt{g/\lambda}/n$.  \end{conjecture}

 \begin{remark} \rm Note that the covering radii of  $\mathcal{N}_G, {\rm Crit}_{\triangle}(L_G)$ and ${\rm Crit}_{\bar{\triangle}}(L_G)$ with respect to any fixed convex body are equal. 
 The covering radius of $\mathcal{N}_G$ with respect to the polytope $P_{1,\lambda}$ being at least $\sqrt{g/\lambda}/n$ is equivalent to its covering radius with respect to $P_{1/\sqrt{\lambda},\sqrt{\lambda}}$ being at least $\sqrt{g}/n$.  Furthermore,  $P_{1/\sqrt{\lambda},\sqrt{\lambda}}=-P_{\sqrt{\lambda},1/\sqrt{\lambda}}$ and the covering radius of $\mathcal{N}_G$ with respect to $P_{\sqrt{\lambda},1/\sqrt{\lambda}}$ and $-P_{\sqrt{\lambda},1/\sqrt{\lambda}}$ are equal (by the reflection invariance property of the set ${\rm Crit}_{\triangle}(L_G)$). Hence, it suffices to prove the covering radius conjecture in the interval $[1,g]$ (or equivalently, $[1/g,1]$)\footnote{We thank the anonymous referee for this remark. }. \qed \end{remark}

{\bf Proof of Equivalence of Conjecture \ref{geomver_conj} and Conjecture \ref{bakconjreal}:}  ($\Leftarrow$) Given a $\lambda \in [1/g,g]$, set $r_0= \sqrt{g/\lambda}-1$ and $d_0= g+(r_0+1)(1-\lambda)-1$.  Note that $r_0,~d_0$ are both non-negative and since $r_0+1 \leq g$ and $(1-\lambda) \leq (g-1)/g$, we have $d_0 \leq 2g-2$. We verify that this pair of $r_0,~d_0$ satisfies $\rho(g,r_0,d_0) = 0$ as follows:
 
 \begin{center}
 
 $\rho(g,r_0,d_0)=g-(r_0+1)(g-d_0+r_0) =g-(r_0+1)^2 \lambda=0$.\\
 \end{center}
 
 Furthermore, we note that $g-d_0+r_0=(r_0+1)\lambda= \sqrt{g \cdot \lambda} \geq 1$.  Hence, by Conjecture \ref{bakconjreal}, there exists a divisor $D$ of degree $d_0$ and rank equal to $r_0$.  In order to apply Proposition \ref{geomrank_prop} to $D$, note that $\lambda=(g-d_0+r_0)/(r_0+1)$ and hence, $P_{\frac{r_0+1}{n},\frac{g-d_0+r_0}{n}}=P_{\frac{r_0+1}{n},\lambda \cdot (\frac{r_0+1}{n})}=(r_0+1)/n \cdot P_{1,\lambda}$.   By Proposition \ref{geomrank_prop}, this implies that $\pi_{g-1}(D)$ in $H_{g-1}$ satisfies $h_{P_{1,\lambda},\mathcal{N}_G}(\pi_{g-1}(D))=(r_0+1)/n=\sqrt{g/\lambda}/n$. Hence, the covering radius of $\mathcal{N}_G$ with respect to $P_{1,\lambda}$ is at least $\sqrt{g/\lambda}/n$. 
 
($\Rightarrow$) Suppose that $(r_0,d_0)$ is a pair of non-negative real numbers such that $\rho(g,r_0,d_0) \geq 0$. Set $\lambda=(g-d_0+r_0)/(r_0+1)$. By Remark \ref{redreal_rem}, we can assume without loss of generality that $d_0 \leq g-1$. This implies that $\lambda \geq 1$.  Furthermore, since $\rho(g,r_0,d_0) \geq 0$, we have $\lambda \leq g/(r_0+1)^2 \leq g$. Hence, $\lambda \in [1,g]$ and  Conjecture \ref{geomver_conj} implies that there exists a point ${\bf p} \in H_{g-1}$ such that $h_{P_{1,\lambda},\mathcal{N}_G}({\bf p}) \geq \sqrt{g/\lambda}/n$. Note that $0 \leq (r_0+1)/n \leq \sqrt{g/\lambda}/n$. 
  Since $h_{P_{1,\lambda},\mathcal{N}_G}: H_{g-1} \rightarrow \mathbb{R}$ is a continuous function (by Remark \ref{hct_rem}) and attains a zero, by the intermediate value theorem for continuous functions, there is a point ${\bf q} \in H_{g-1}$ such that $h_{P_{1,\lambda},\mathcal{N}_G}({\bf q})=(r_0+1)/n$.  Next, we lift it to an $\mathbb{R}$-divisor with prescribed degree and rank as follows:  consider the $\mathbb{R}$-divisor $D_{\mathbb{R}}={\bf q}-(g-1-d_0)/n \cdot (1,\dots,1)$. The degree of $D_{\mathbb{R}}$ is equal to $d_0$ and   by Proposition \ref{geomrank_prop}, it has rank $r_0$. \qed
 
%


\subsection{Brill-Noether Existence for $\mathbb{R}$-Divisors on Graphs}\label{bndenseproof_subsect}


Our main goal in this section is to prove Theorem \ref{bnreal_intro}.  We start with a description of the polytope $P_{\alpha,\bar{\alpha}}$ for positive $\alpha$ and $\bar{\alpha}$.   Recall that the vertices of $\triangle$ are denoted by ${\bf b_1},\dots,{\bf b_n}$. For real numbers $\alpha,~\bar{\alpha}$,  let ${\bf w_{i,j}}=\alpha{\bf b}_i-\bar{\alpha}{\bf b}_j$ for $i,~j$ from $1,\dots,n$.  The coordinates of ${\bf w}_{i,j}$ (assuming $i \neq j$) are as follows.

\begin{center}
$({\bf w}_{i,j})_k=
\begin{cases}
-\alpha+\bar{\alpha},  \text{~ if $k \notin \{i,j\}$},\\
\alpha(n-1)+\bar{\alpha}, \text{~if $k=i$},\\
-\bar{\alpha}(n-1)-\alpha, \text{~if $k=j$}.
\end{cases}
$\\
\end{center}


\begin{proposition}\label{verdesc_prop}
For any pair of positive real numbers $\alpha,~\bar{\alpha}$. 
The polytope $P_{\alpha,\bar{\alpha}}$ has $n(n-1)$ vertices and they are $\{{\bf w}_{i,j}\}_{i \neq j}$  over all pairs of integers $i,~j$ from $1,\dots,n$. 
\end{proposition}

\begin{proof}

For any pair of polytopes $Q_1,~Q_2$, the vertices of the Minkowski sum are of the form ${\bf v}+{\bf v'}$ where ${\bf v}$ is a vertex of $Q_1$ and ${\bf v'}$ is a vertex of $Q_2$. Hence, every vertex of $P_{\alpha,\bar{\alpha}}$ is of the form ${\bf w}_{i,j}$ for some $i,j$. 

In the following, we show that every ${\bf w}_{i,j}$ for $i \neq j$ is a vertex.  We consider the linear functional $\ell_{i,j}({\bf p})=p_i-p_j$ where ${\bf p}=(p_1,\dots,p_n)$. We note that $\ell_{i,j}({\bf w}_{i,j})=n(\alpha+\bar{\alpha}),~\ell_{i,j}({\bf w}_{j,i})=-n(\alpha+\bar{\alpha}),~\ell_{i,j}({{\bf w}_{i,k}})=\alpha n,~\ell_{i,j}({{\bf w}_{k,i}})=-\bar{\alpha} n,~\ell_{i,j}({{\bf w}_{s,j}})=\bar{\alpha} n,~~\ell_{i,j}({{\bf w}_{j,s}})=-\alpha n$ for any $s \neq i,~k \neq j$ and $\ell_{i,j}({\bf w}_{s,k})=0$ if $s,k \notin \{i,j\}$.  Hence, $\ell_{i,j}$ is uniquely maximised at ${\bf w}_{i,j}$ and we conclude that ${\bf w}_{i,j}$ is a vertex of $P_{\alpha,\bar{\alpha}}$.

We are left with showing that ${\bf w}_{i,i}$ is not a vertex of $P_{\alpha,\bar{\alpha}}$. To see this, we note that ${\bf w}_{i,i}=(\alpha-\bar{\alpha}){\bf b_i}$.  If $\alpha=\bar{\alpha}$, then ${\bf w}_{i,i}$ is the origin $O$ and is not a vertex since $({\bf w}_{i,j}+{\bf w}_{j,i})/2=O$ for any $i \neq j$.  
Suppose $\alpha>\bar{\alpha}$,  consider the point $(\sum_{j \neq i}{\bf w_{i,j}})/(n-1)$ in $P_{\alpha,\bar{\alpha}}$ for some $i$. Since  $\sum_{k}{{\bf b}_k}=O$, we obtain $\sum_{j \neq i}{\bf w_{i,j}}/(n-1)=(\alpha+\bar{\alpha}/(n-1)){\bf b_i}$. Note that since $\alpha \geq \bar{\alpha}>0$ we have $0 < \alpha-\bar{\alpha} <\alpha< \alpha+\bar{\alpha}/(n-1)$ and that the origin is contained in $P_{\alpha,\bar{\alpha}}$. Hence, ${\bf w}_{i,i}$ is contained in the relative interior of the line segment defined by the origin and this point. Hence, it is not a vertex.  The case where $\bar{\alpha}>\alpha$ follows by a similar argument: consider the point  $(\sum_{j \neq i}{\bf w_{j,i}})/(n-1)$ and show that ${\bf w}_{i,i}$ is contained in the relative interior of the origin and this point.
\end{proof}

In order to obtain a lower bound for $h_{P_{\alpha,\bar{\alpha}},\mathcal{N}_G}$ in terms of the corresponding simplicial distance functions, we compute $d_{\triangle}(O,{\bf w}_{i,j})$ and $d_{\bar{\triangle}}(O,{\bf w}_{i,j})$ at the vertices of $P_{\alpha,\bar{\alpha}}$.

\begin{lemma}\label{distsimver_lem}
For positive real numbers $\alpha,~\bar{\alpha}$, let ${\bf w}_{i,j}$ be a vertex of $P_{\alpha,\bar{\alpha}}$.  The simplicial distance from the origin to ${\bf w}_{i,j}$ is given by $d_{\triangle}(O,{\bf w}_{i,j})=\bar{\alpha}(n-1)+\alpha$ and $d_{\bar{\triangle}}(O,{\bf w}_{i,j})=\alpha(n-1)+\bar{\alpha}$.
\end{lemma}
\begin{proof}

 By Remark \ref{simpdisform_rem}, the simplicial distance functions are given by the formulas $d_{\triangle}({\bf p},{\bf q})=|{\rm min}_i (q_i-p_i)|$ and  $d_{\bar{\triangle}}({\bf p},{\bf q})=|{\rm max}_i (q_i-p_i)|$ for any ${\bf p}=(p_1,\dots,p_n)$ and ${\bf q}=(q_1,\dots,q_n)$ in $H_0$. Since $n \geq 2$ and $\alpha,~\bar{\alpha}$ are non-negative, ${\rm min}_k ({\bf w}_{i,j})_k=-\bar{\alpha}(n-1)-\alpha$ and ${\rm max}_k ({\bf w}_{i,j})_k=\alpha(n-1)+\bar{\alpha}$. Using the formulas for the simplicial distance functions completes the proof.
\end{proof}

\begin{proposition}
For any point ${\bf p} \in H_0$, the distance function $d_{P_{\alpha,\bar{\alpha}}}$ at $(O,{\bf p})$ is lower bounded by $d_{\triangle}(O,{\bf p})$ and $d_{\bar{\triangle}}(O,{\bf p})$ as follows: 

\begin{center}
$d_{P_{\alpha,\bar{\alpha}}}(O,{\bf p}) \geq d_{\triangle}(O,{\bf p})/(\bar{\alpha}(n-1)+\alpha)$,\\
$d_{P_{\alpha,\bar{\alpha}}}(O,{\bf p}) \geq d_{\bar{\triangle}}(O,{\bf p})/(\alpha(n-1)+\bar{\alpha})$.

\end{center}

 These lower bounds are tight, i.e. there are points ${\bf q_1},~{\bf q_2} \in H_0$ such that $d_{P_{\alpha,\bar{\alpha}}}(O,{\bf q_1})=d_{\triangle}(O,{\bf q_1})/(\bar{\alpha}(n-1)+\alpha)$ and $d_{P_{\alpha,\bar{\alpha}}}(O,{\bf q_2})=d_{\bar{\triangle}}(O,{\bf q_2})/(\alpha(n-1)+\bar{\alpha})$.

\end{proposition}

\begin{proof}
Consider the following optimisation problems: ${\rm max}_{{\bf r} \in P_{\alpha,\bar{\alpha}}}d_{\triangle}(O,{\bf r})$ and ${\rm max}_{{\bf r} \in P_{\alpha,\bar{\alpha}}}d_{\bar{\triangle}}(O,{\bf r})$.  In the following, we note that a solution is attained at a vertex.  Suppose that ${\bf s} \in P_{\alpha,\bar{\alpha}}$ then ${\bf s}=\sum_i \lambda_i \cdot {\bf v}_i$ where $\lambda_i \geq 0$ for all $i$, $\sum_i \lambda_i=1$ and each ${\bf v}_i$ is a vertex of $P_{\alpha,\bar{\alpha}}$. By the triangle inequality property of $d_{\triangle}$ and $d_{\bar{\triangle}}$ (and their scaling equality), we have $d_{\triangle}(O,{\bf s}) \leq  \sum_i \lambda_i d_{\triangle}(O,{\bf v_i})$ and $d_{\bar{\triangle}}(O,{\bf s}) \leq  \sum_i \lambda_i d_{\bar{\triangle}}(O,{\bf v_i})$.  Hence, $d_{\triangle}(O,{\bf s}) \leq {\rm max}_i~ d_{\triangle}(O,{\bf v_i})$ and $d_{\bar{\triangle}}(O,{\bf s}) \leq {\rm max}_i~d_{\bar{\triangle}}(O,{\bf v_i})$. By Proposition \ref{verdesc_prop}, the vertices of $P_{\alpha,\bar{\alpha}}$ are of the form ${\bf w}_{i,j}$ for $i \neq j$.
 By Lemma \ref{distsimver_lem},  we have $d_{\triangle}(O,{\bf w}_{i,j})=\bar{\alpha}(n-1)+\alpha$ and $d_{\bar{\triangle}}(O,{\bf w}_{i,j})=\alpha(n-1)+\bar{\alpha}$ for any vertex of $P_{\alpha,\bar{\alpha}}$. Hence, ${\rm max}_{{\bf r} \in P_{\alpha,\bar{\alpha}}}d_{\triangle}(O,{\bf r})=\bar{\alpha}(n-1)+\alpha$ and ${\rm max}_{{\bf r} \in P_{\alpha,\bar{\alpha}}}~d_{\bar{\triangle}}(O,{\bf r})=\alpha(n-1)+\bar{\alpha}$.  For  a point ${\bf p} \in H_0$ such that $d_{P_{\alpha,\bar{\alpha}}}(O,{\bf p})=d$,  there is a unique point ${\bf s} \in P_{\alpha,\bar{\alpha}}$ such that ${\bf p}=d \cdot {\bf s}$. Since, $d_{\triangle}$ and $d_{\bar{\triangle}}$  both satisfy the scaling equality, we have:  \begin{center}$d_{\triangle}(O,{\bf p})=d \cdot d_{\triangle}(O,{\bf s}) \leq d \cdot (\bar{\alpha}(n-1)+\alpha)$ and $d_{\bar{\triangle}}(O,{\bf p})=d \cdot d_{\bar{\triangle}}(O,{\bf s}) \leq d \cdot (\alpha(n-1)+\bar{\alpha})$. \end{center}   This completes the proof. \end{proof}

This leads us to the following norm conversion inequality that lower bounds the distance function $h_{P_{\alpha,\bar{\alpha}},T}$ in terms of the distance functions $h_{\triangle,T}$ and $h_{\bar{\triangle},T}$ for any discrete subset $T$ of $H_d$.

\begin{corollary} \label{covradlb_cor}{\rm { \bf(Norm Conversion Inequality)}}
Fix an integer $d$. Let $T$ be a discrete subset of $H_d$ carrying the action (by translation) of a full rank lattice in $H_0$. For any point ${\bf p} \in H_d$:

\begin{center}
$h_{P_{\alpha,\bar{\alpha}},T}({\bf p}) \geq h_{\triangle,T}({\bf p})/(\bar{\alpha}(n-1)+\alpha)$,\\
$h_{ P_{\alpha,\bar{\alpha}},T}({\bf p}) \geq  h_{\bar{\triangle},T}({\bf p})/(\alpha(n-1)+\bar{\alpha})$,\\
${\rm Cov}_{P_{\alpha,\bar{\alpha}}}(T) \geq {\rm max}\{{\rm Cov}_{\triangle}(T)/(\bar{\alpha}(n-1)+\alpha),{\rm Cov}_{\bar{\triangle}}(T)/(\alpha(n-1)+\bar{\alpha})\}$
\end{center}
where ${\rm Cov}_C(T)$ is the covering radius of $T$ with respect to the convex body $C$.
\end{corollary}

In the following, we show Theorem \ref{bnreal_intro} and the covering radius conjecture for the case of dense multigraphs, i.e. $\kappa=1$. 

\begin{proposition}\label{bndense_intro}
Theorem \ref{bnreal_intro} and the covering radius conjecture (Conjecture \ref{geomver_conj}) hold for any dense multigraph. Hence, the Brill-Noether existence conjecture for $\mathbb{R}$-divisors (Conjecture \ref{bakconjreal}) also holds for any dense multigraph. 
\end{proposition}

\begin{proof}
We specialise Corollary \ref{covradlb_cor} to $T=\mathcal{N}_G$ and by Proposition \ref{covcritsimpcru_prop} and Proposition \ref{covradeq_prop}, we have ${\rm Cov}_{\triangle}(\mathcal{N}_G)={\rm Cov}_{\bar{\triangle}}(\mathcal{N}_G)=m/n$. Hence,  
\begin{center} ${\rm Cov}_{P_{\alpha,\bar{\alpha}}}(T) \geq m/(n({\rm min}(\bar{\alpha}(n-1)+\alpha,\alpha(n-1)+\bar{\alpha})))$.\end{center} 

Substituting $\alpha=1$ and $\bar{\alpha}=\lambda$, we obtain 

\begin{equation}\label{covradlb_eq}
{\rm Cov}_{P_{\alpha,\bar{\alpha}}}(T) \geq m/(n({\rm min}( \lambda (n-1)+1,(n-1)+\lambda))).
\end{equation}

Next, we show that the inequality $m/(n({\rm min}( \lambda (n-1)+1,(n-1)+\lambda))) \geq \sqrt{g}/(n\sqrt{\lambda})$ where $\lambda \in [1/g,g]$ holds for all dense multigraphs $G$.  We start by noting that this inequality is equivalent to $m/\sqrt{g} \geq {\rm min}(\sqrt{\lambda}(n-1)+1/\sqrt{\lambda},(n-1)/\sqrt{\lambda}+\sqrt{\lambda})$.  Since both sides of this inequality are fixed by the substitution $\lambda \rightarrow 1/\lambda$. It suffices to show the inequality for $\lambda \in [1,g]$. In this case, we show that 
$m/\sqrt{g} \geq (n-1)/\sqrt{\lambda}+\sqrt{\lambda}$. Plugging, $m=g+n-1$ we obtain: $\sqrt{g}+(n-1)/\sqrt{g} \geq  (n-1)/\sqrt{\lambda}+\sqrt{\lambda}$ which is equivalent to  $\sqrt{g}-\sqrt{\lambda} \geq (n-1)(1/\sqrt{\lambda}-1\sqrt{g})$. Using the fact that $\lambda \leq g$ (and hence, $\sqrt{\lambda} \leq \sqrt{g}$), this in turn is equivalent to $\sqrt{g \lambda} \geq n-1$. This holds since $G$ is dense (and hence, $g\geq (n-1)^2$) and $\lambda \geq 1$.  This shows that the covering radius conjecture  for dense multigraphs. 
By its equivalence with Conjecture \ref{bakconjreal},  we conclude that Conjecture \ref{bakconjreal} also holds for dense multigraphs. 
We finally note that  $\kappa=1$ for any dense multigraph and hence,  Theorem \ref{bnreal_intro} also holds\footnote{We thank the anonymous referee for suggesting this argument.}.\end{proof}


Note that by Proposition \ref{covradeq_prop} we have also shown Conjecture \ref{geomver_conj} for ${\rm Crit}_{\triangle}(L_G)$ and ${\rm Crit}_{\bar{\triangle}}(L_G)$.   In the following, we go over the proof technique of  Proposition \ref{bndense_intro} for multiples of complete graphs.

\begin{example}
\rm Consider multigraphs of the form $\beta \cdot K_n$ for $\beta \geq 2$.  As mentioned in Example \ref{multikn_ex}, they are dense. By Proposition \ref{covcritsimp_prop}, we know that the covering radius of ${\rm Crit}_{\triangle}(L_G)$ with respect to $\triangle$ (and $\bar{\triangle}$) is attained at the origin since the canonical divisor $K_{\beta \cdot K_n}$ of $\beta \cdot K_n$ is $(\beta(n-1)-2,\dots,\beta(n-1)-2)$ (where the canonical divisor is regarded as a point in $\mathbb{Z}^n$)  and hence, $\pi_0(K_{\beta \cdot K_n})$ is the origin.  We then used norm conversion inequalities to obtain the lower bound $\sqrt{g/\lambda}/n$ on the covering radius of ${\rm Crit}_{\triangle}(L_G)$  with respect to $P_{1,\lambda}$.  We believe that the covering radius of  ${\rm Crit}_{\triangle}(L_G)$ with respect to  $P_{1,\lambda}$ is also attained at the origin.  \qed 

\end{example}

 \begin{remark} \rm  For dense multigraphs satisfying the stronger condition  $g \geq n^2$, Conjecture \ref{bakconjreal_intro}  can also be proven as follows.  Let $(r_0,d_0)$ be a pair of non-negative real numbers such that $\rho(g,r_0,d_0) \geq 0$.   By Remark \ref{redreal_rem}, we assume that $d_0 \leq g-1$. Consider the $\mathbb{R}$-divisor  $D_{\mathbb{R}}=r_0 \cdot (1,\dots,1)$. Note $D_{\mathbb{R}}$ has rank at least $r_0$ and degree $nr_0$ (see Item 1, Remark \ref{ranksub_rem} of Appendix \ref{rrreal_sect}). On the other hand, since $\rho(g,r_0,d_0) \geq 0$, we have $d_0 \geq g-g/(r_0+1)+r_0$. Furthermore, $d_0 \leq g-1$ and $\rho(g,r_0,d_0) \geq 0$ implies that  $r_0+1 \leq \sqrt{g}$. Combining the two inequalities, yields $d_0 \geq r_0 \sqrt{g}$.  By assumption, $g \geq n^2$. Hence, $\sqrt{g} \geq n$ and  $d_0 \geq  r_0 n$. Subtracting a suitable non-negative real multiple of $(1,0,\dots,0)$ yields Conjecture \ref{bakconjreal_intro}.   If $r_0$ and $d_0$ are non-negative integers, a similar argument also shows Conjecture \ref{bakconjv1} for dense multigraphs\footnote{We thank the anonymous referee for this remark.}. 


However, even in this case the lower bound for the rank of a divisor of fixed degree $r_0n$ obtained in the proof of Proposition \ref{bndense_intro} via norm conversion inequalities generalises and improves the lower bound $r_0$ obtained by the above method.
To see this, we set $\lambda_0:=(g-r_0n+r_0)/(r_0+1)=m/(r_0+1)-(n-1)$ and apply Inequality (\ref{covradlb_eq}) to obtain 
\begin{center}

${\rm Cov}_{P_{1, \lambda_0}}(\mathcal{N}_G) \geq \dfrac{m}{n({\rm min}(\dfrac{m(n-1)}{(r_0+1)}-(n-1)^2+1,\dfrac{m}{(r_0+1)}))}$.

\end{center}

An argument similar to the proof of Conjecture \ref{geomver_conj} implies Conjecture \ref{bakconjreal} implies that there exists an $\mathbb{R}$-divisor of degree $nr_0$ and rank at least \begin{center} $\dfrac{m}{{\rm min}(\dfrac{m(n-1)}{(r_0+1)}-(n-1)^2+1,\dfrac{m}{(r_0+1)})}-1 \geq r_0$. \end{center}

 This lower bound is strictly greater than $r_0$ if and only if the minimum in the denominator is attained only in the first term. This in turn is equivalent to $m<n \cdot (r_0+1)$. Asymptotically in $n$, this happens, for instance, if $m \in O(n^k)$ and $r_0 \in \Omega(n^{k-1+\epsilon})$ for some $k \geq 1$ and $\epsilon>0$. \qed
\end{remark}



Recall that the \emph{stretch factor} $\kappa$ of a graph with $n$ vertices and $m$ edges is defined to be the maximum of $\lceil (n(n-1)/m \rceil$ and one.  We obtain the following weak version of the covering radius conjecture for arbitrary undirected, connected multigraphs.


\begin{corollary}\label{covlow_cor} {\rm ({\bf Lower Bound on the Covering Radius})}
Let $G$ be an undirected, connected multigraph on $n \geq 2$ vertices and of genus $g \geq 1$. Let $\lambda \in [1/g,g]$. The covering radius of ${\rm Crit}_{\triangle}(L_G)$ with respect to $P_{1,\lambda}$ is at least $\dfrac{(\sqrt{g/\kappa\lambda})}{n}$.
\end{corollary}

\begin{proof} By the definition of stretch factor, the multigraph $\kappa \cdot G$ is dense and hence, by Proposition  \ref{bndense_intro} (its covering radius version) the covering radius of ${\rm Crit}_{\triangle}( L_{\kappa \cdot G})$ with respect to $P_{1,\lambda}$ is at least $\dfrac{(\sqrt{g(\kappa \cdot G)/\lambda})}{n}$, where $g(\kappa \cdot G)$  is the genus of $\kappa \cdot G$. By the definition of the Laplacian lattice, we have $L_{\kappa \cdot G}=\kappa \cdot L_{G}$ and hence, $h_{\triangle,L_{\kappa \cdot G}}(\kappa \cdot {\bf p})=\kappa \cdot h_{\triangle,L_G}({\bf p})$ for all ${\bf p} \in H_0$. Hence, using the description of ${\rm Crit}_{\triangle}(L_{\kappa \cdot G})$ and ${\rm Crit}_{\triangle}(L_{G})$  as local maxima of $h_{\triangle,L_{\kappa \cdot G}}$ and $h_{\triangle,L_G}$ respectively, we have ${\rm Crit}_{\triangle}(L_{\kappa \cdot G})=\kappa \cdot {\rm Crit}_{\triangle}(L_{G})$.  For $\lambda \in [1/g(\kappa \cdot G), g(\kappa \cdot G)]$ we have ${\rm Cov}_{P_{1,\lambda}}({\rm Crit}_{\triangle}(L_{\kappa \cdot G}))=\kappa \cdot {\rm Cov}_{P_{1,\lambda}}({\rm Crit}_{\triangle}(L_{G})) \geq \dfrac{\sqrt{g(\kappa \cdot G)/\lambda}}{n}$. This gives \begin{center} ${\rm Cov}_{P_{1,\lambda}}({\rm Crit}_{\triangle}(L_{G}))  \geq \dfrac{\sqrt{g(\kappa \cdot G)/\lambda}}{(\kappa \cdot n)}\geq  \dfrac{(\sqrt{g/\kappa\lambda})}{n}$.\end{center} The last inequality uses the observation that $g(\kappa \cdot G) \geq \kappa \cdot g$. Since $g(\kappa \cdot G) \geq g$, we have $[1/g,g] \subseteq [1/g(\kappa \cdot G),g(\kappa \cdot G)]$. \end{proof}

\begin{example} \rm
The stretch factor of the Petersen graph $K_{5,2}$ is six, $g=6$ and $n=10$. By Corollary \ref{covlow_cor}, its covering radius with respect to $P_{1,\lambda}$ is at least $\sqrt{1/ \lambda}/10$.  On the other hand, the covering radius conjecture predicts a lower bound of $\sqrt{6/ \lambda}/10$. \qed
\end{example}


Translating Corollary \ref{covlow_cor} into the language of the existence conjecture for $\mathbb{R}$-divisors gives us Theorem \ref{bnreal_intro}. The proof for Corollary \ref{covlow_cor} implies Theorem \ref{bnreal_intro} is similar to the proof that the covering radius conjecture (Conjecture \ref{geomver_conj}) implies the existence conjecture for $\mathbb{R}$-divisors  (Conjecture \ref{bakconjreal}).    In the following, we provide the details for the sake of completeness.

{\bf Proof of Theorem \ref{bnreal_intro}:}  Suppose that non-negative integers $r_0,~d_0$ satisfy $\tilde{\rho}(g,r_0,d_0) \geq 0$. Also, suppose that $d_0 \leq g-1$. Let $\lambda=(g-d_0+r_0)/(r_0+1)$. Since $d_0 \leq g-1$, we have $\lambda \geq 1$ and since $\tilde{\rho}(g,r_0,d_0) \geq 0,~\kappa \geq 1$ and $r_0 \geq 0$, we have $\lambda \leq g/(\kappa(r_0+1)^2) \leq g$. By Corollary  \ref{covlow_cor}, the covering radius of ${\rm Crit}_{\triangle}(L_G)$ with respect to $P_{1,\lambda}$ is at least $\dfrac{(\sqrt{g/\kappa\lambda})}{n}$. Hence, there exists a point ${\bf p} \in H_{g-1}$ such that $h_{P_{1,\lambda},\mathcal{N}_G}({\bf p}) \geq \dfrac{(\sqrt{g/\kappa\lambda})}{n}$. Note that $0 \leq (r_0+1)/n  \leq \dfrac{(\sqrt{g/\kappa\lambda})}{n}$.  Furthermore, $h_{P_{1,\lambda},\mathcal{N}_G}: H_{g-1} \rightarrow \mathbb{R}$ is a continuous function (Remark \ref{hct_rem}) and attains a zero. By the intermediate value theorem for continuous functions, there is a point ${\bf q} \in H_{g-1}$ such that $h_{P_{1,\lambda},\mathcal{N}_G}({\bf q})=(r_0+1)/n$. We lift it to an $\mathbb{R}$-divisor with prescribed degree and rank as follows:  consider the $\mathbb{R}$-divisor $D_{\mathbb{R}}={\bf q}-(g-1-d_0)/n \cdot (1,\dots,1)$. The degree of $D_{\mathbb{R}}$ is equal to $d_0$ and   by Proposition \ref{geomrank_prop}, it has rank $r_0$.  The case $d_0>g-1$ follows from Remark \ref{redreal_rem}.\qed

For instance, the stretch factor of the complete graph is equal to two and hence, Theorem \ref{bnreal_intro} specialises to the following statement.

\begin{corollary}
Consider any integer $n \geq 2$, and non-negative reals $r$ and $0 \leq d \leq 2 \binom{n-1}{2}-2$. If $\binom{n-1}{2}-2(r+1)(\binom{n-1}{2}-d+r) \geq 0$, then there exists an  $\mathbb{R}$-divisor of degree at most $d$ and rank equal to $r$ on $K_n$.
\end{corollary}

Next, we derive an upper bound on the gonality sequence of dense graphs. For an integer $k \geq 1$, the \emph{$k$-th $\mathbb{R}$-gonality} $\gamma_{k,\mathbb{R}}(G)$ is defined as the infimum over the degree of all $\mathbb{R}$-divisors of rank at least $k$.





\begin{corollary}\label{realgon_cor}
 If $G$ has stretch factor $\kappa$ then $\gamma_{k,\mathbb{R}}(G)  \leq \lceil g(\kappa(k+1)-1)/(\kappa(k+1))+k  \rceil$ for all integers $k \geq 1$.
\end{corollary}

\begin{proof}
Consider the pair $(r_0,d_0)=(k,\lceil g(\kappa(k+1)-1)/(\kappa(k+1))+k \rceil)$ and note that $\tilde{\rho}(g,r_0,d_0) \geq 0$.  If  $k \leq g-2$, then $d_0 \leq 2g-2$ and Theorem \ref{bnreal_intro} implies the claim. Otherwise, if $k \geq g-1$ then 
Riemann-Roch provides an upper bound of $g-1+k$ on $\gamma_{k,\mathbb{R}}(G)$  and  $g-1+k \leq  \lceil g(\kappa(k+1)-1)/(\kappa(k+1))+k  \rceil$.  \end{proof}


\begin{remark}
\rm The ceiling in the upper bound in Corollary \ref{realgon_cor} can be removed for certain ``small'' values of $k$ (if $g(\kappa(k+1)-1)/(\kappa(k+1))+k \leq 2g-2$). With Remark \ref{realhighdeg_rem} in mind, we stated an upper bound that is valid for all integers $k \geq 1$. \qed
\end{remark}

\section{Brill-Noether Existence for Divisors on $G$}

In this section, we study the existence conjecture for graphs (Conjecture \ref{bakconjv1}) using techniques similar to the ones in Subsection \ref{bndenseproof_subsect}.

\subsection{An Approximate Version of  Brill-Noether Existence}\label{round_subsect}

We use Theorem \ref{bnreal_intro} to prove Theorem \ref{bnapprox_intro}.  The main idea of the proof is to ``round-off'' a divisor produced by applying Theorem \ref{bnreal_intro}.  

{\bf Proof of Theorem \ref{bnapprox_intro}:} Suppose that $(r_0,d_0)$ is a pair of integers  (with $0 \leq  d_0 \leq 2g-2$) satisfying $\tilde{\rho}(g,r_0,d_0) \geq 0$.  We apply Theorem \ref{bnreal_intro} to $(r_0,d_0)$ to obtain an $\mathbb{R}$-divisor $D_{\mathbb{R}}=\sum_v d_v (v)$ on $G$ with degree $d'_0$ at most $d_0$ and rank $r_0$. 
We round-off the $\mathbb{R}$-divisor $D_{\mathbb{R}}$ to a divisor with the prescribed properties as follows.  Fix any vertex $v_0$ of $G$ and define a divisor $D=\sum_{v \neq v_0} [d_v](v)+(d'_0-\sum_{v \neq v_0}[d_v])(v_0)$ where $[.]$ is a nearest integer function. By construction, the divisor $D$ has degree $d'_0$.  

Let $E=\sum_{v \neq v_0} (v)+(n-1)(v_0)$. Note that $D_{\mathbb{R}}-\nu+E \geq D-\nu$ and $D-\nu+E \geq D_{\mathbb{R}}-\nu$ for any $\nu \in \mathcal{N}_G$.  Furthermore, note that ${\rm deg^{+}}(D'+E') \leq {\rm deg}^{+}(D')+{\rm deg}(E')$ for any divisor $D'$ and any effective divisor $E'$.   We apply this twice, first with $D'=D_{\mathbb{R}}-\nu$,~$E'=E$ and  second with $D'=D-\nu,~E'=E$ to obtain $|{\rm deg^{+}}(D_{\mathbb{R}}-\nu)-{\rm deg^{+}}(D-\nu)| \leq 2n-2$ for every $\nu \in \mathcal{N}_G$ (note that ${\rm deg}(E)=2n-2)$).   Hence,  ${\rm min}_{\nu \in \mathcal{N}_G}{\rm deg^{+}}(D_{\mathbb{R}}-\nu) \leq  {\rm min}_{\nu \in \mathcal{N}_G}{\rm deg^{+}}(D-\nu)+2n-2$ and   ${\rm min}_{\nu \in \mathcal{N}_G}{\rm deg^{+}}(D-\nu) \leq  {\rm min}_{\nu \in \mathcal{N}_G}{\rm deg^{+}}(D_{\mathbb{R}}-\nu)+2n-2$.  From the degree plus formula for the rank of an $\mathbb{R}$-divisor (Proposition \ref{degplusrank_theo}), we obtain $|r_0-r(D)|=|r(D_{\mathbb{R}})-r(D)| \leq 2n-2$. \qed

\subsection{Integral Covering Radius Conjecture}\label{intcovrad_subsect}

In the following, we formulate the existence conjecture for $d=g-1$ in the terms of a variant of the covering radius called the \emph{integral covering radius} and prove it for sufficiently dense graphs. 


\begin{definition}{\rm ({\bf Integral Covering Radius})}\label{intcovrad_def}
Let $\Lambda$ be a full rank  sublattice of $H_0 \cap \mathbb{Z}^n$.  Let $T$ be a subset of  $H_d$ carrying the action of $\Lambda$ via translation and with finitely many orbits. Let $C$ be a convex body in $H_0$.  The maximum value of the function $h_{C,T}| (H_{d} \cap \mathbb{Z}^n)$ ($h_{C,T}$ is as in Section \ref{covrad_sect}) is called the integral covering radius of $T$ with respect to $C$ and is denoted by ${\rm Covint}_C(T)$. 
\end{definition}


\begin{remark}
\rm
Unlike the covering radius, the integral covering radius of ${\rm Crit}_{\triangle}(L_G),~{\rm Crit}_{\bar{\triangle}}(L_G)$ and $\mathcal{N}_G$ are a priori different. However, we do not know of explicit examples where they are different.  \qed
\end{remark}

We are now ready to state a geometric reformulation of the existence conjecture in the case $d=g-1$. 

\begin{conjecture}\label{geomverell1_conj} {\rm {(\bf Integral Covering Radius Conjecture)}}
 Let $P_1$ be the unit ball of the $\ell_1$-norm in $H_0$.  The integral covering radius of $\mathcal{N}_G$ with respect to $P_1$ is at least $2\lfloor\sqrt{g} \rfloor$.
\end{conjecture}

 Note that the polytope $P_1$ is homothetic to $P_{1,1}=\triangle+\bar{\triangle}$, more precisely $P_{1,1}=2 n \cdot P_1$.  Furthermore,  since with respect to a fixed convex body $C$, the integral covering radius is at most the covering radius of $\mathcal{N}_G$, this implies that the covering radius of $\mathcal{N}_G$ with respect to $P_1$ is at least $2\lfloor\sqrt{g} \rfloor$ (this is consistent with the covering radius conjecture).  In the following, we note that Conjecture \ref{geomverell1_conj} is equivalent to the  existence conjecture for $d=g-1$. The following statement is a corollary to Proposition \ref{degplusrank_theo}


\begin{corollary}\label{degplus_corr}
If $D$ is a divisor of degree $g-1$, then $r(D)={\rm min}_{\nu \in \mathcal{N}_G} |D-\nu|_1/2-1$. 
\end{corollary}

{\bf Proof of Equivalence of  the Integral Covering Radius Conjecture and the Existence Conjecture (Conjecture \ref{bakconjv1}) for $d=g-1$}:  ($\Rightarrow$)The integral covering radius of  $\mathcal{N}_G$ with respect to $P_1$ being at least $2\lfloor\sqrt{g} \rfloor$ implies that there exists an integer point (a point ${\bf p} \in H_{g-1} \cap \mathbb{Z}^n$) such that ${\rm min}_{\nu \in \mathcal{N}_G}|{\bf p}-\nu|_1 \geq 2\lfloor\sqrt{g} \rfloor$. By  Corollary \ref{degplus_corr}, the rank of ${\bf p}$ (considered as a divisor on $G$) is at least $\lfloor\sqrt{g} \rfloor-1$.  Hence, $G$ has a divisor of degree $g-1$ and rank at least $\lfloor \sqrt{g} \rfloor -1$. Next, note that the maximum integer $r$ for which $\rho(g,r,g-1) \geq 0$ is $\lfloor \sqrt{g} \rfloor -1$. 
Hence, given any non-negative integer $r$ such that $\rho(g,r,g-1) \geq 0$ there exists a divisor of degree $g-1$ and rank at least $r$ on $G$. By suitably subtracting an effective divisor from $G$,  we obtain a divisor of degree at most $g-1$ and rank $r$. 

($\Leftarrow$)
Suppose that the existence conjecture holds for $d=g-1$ then there is a divisor of degree $g-1$ and rank at least $\lfloor \sqrt{g} \rfloor -1$ (by suitably adding an effective divisor to a divisor given by Conjecture \ref{bakconjv1} for the pair $(r,d)=(\lfloor \sqrt{g} \rfloor -1,g-1)$). By Corollary \ref{degplus_corr}, the corresponding point ${\bf p}$ in  $H_{g-1} \cap \mathbb{Z}^n$ satisfies  ${\rm min}_{\nu \in \mathcal{N}_G}|{\bf p}-\nu|_1 \geq 2\lfloor\sqrt{g} \rfloor$. Hence, the integral covering radius conjecture holds.  \qed


\subsection{Proof of Theorem \ref{bnKn_theo}}


 Using the formula for the covering radius of ${\rm Crit}_{\triangle}(L_G)$ and $\mathcal{N}_G$ with respect to $\triangle$ and $\bar{\triangle}$ along with the norm conversion inequality (Corollary \ref{covradlb_cor}) and the observation that $P_{1,1}=2 n \cdot P_1$, we obtain the following proposition.

\begin{proposition}\label{covcritell1_prop} The covering radius of ${\rm Crit}_{\triangle}(L_G)$ and $\mathcal{N}_G$ with respect to $P_1$ are both at least $2m/n$, where $m$ is the number of edges and $n$ is the number of vertices of $G$. \end{proposition}


\begin{corollary}\label{intcov_cor} The integral covering radius of $\mathcal{N}_G$ with respect to $P_1$ is at least $2m/n-n/2$.\end{corollary}

\begin{proof} Suppose that $h_{P_1,\mathcal{N}_G}$ attains a maximum at a point ${\bf p}=(p_1,\dots,p_n) \in H_{g-1}$.  Consider $[{\bf p}]=([p_1],\dots,[p_n])  \in \mathbb{Z}^n$ where $[.]$ is a nearest integer function. By the triangle inequality, $|[{\bf p}]-\nu|_1 \geq |{\bf p}- \nu |_1-|{\bf p}-[{\bf p}]|_1 \geq |{\bf p}- \nu |_1-n/2$, for every $\nu \in \mathcal{N}_G$. Hence, by Proposition \ref{covcritell1_prop} we obtain $h_{P_1,\mathcal{N}_G}([{\bf p}]) \geq 2m/n-n/2$. \end{proof}




{\bf Proof of Theorem \ref{bnKn_theo}:}   Since $m=\binom{n}{2}$, Corollary \ref{intcov_cor} implies that ${\rm Cov}_{P_1}(\mathcal{N}_G) \geq (n-2)/2$. Using the fact that $g=\binom{n-1}{2}$ and $n \geq 3$, we obtain ${\rm Cov}_{P_1}(\mathcal{N}_G) \geq (n-2)/2 \geq \sqrt{g}/2$. We then apply Corollary \ref{degplus_corr} to complete the proof of the first part. 


For the second part we note, using Proposition \ref{covradeq_prop} that $h_{\triangle,\mathcal{N}_G}$ attains its maximum over the coset $\pi_{g-1}(K_G)+L_G$. For the complete graph $K_n$, the canonical divisor $K_G=(n-3,,\dots,n-3)$ and $\pi_{g-1}(K_G)= \frac{(n-3)}{2} \cdot (1,\dots,1)$.  If $n$ is odd, $\pi_{g-1}(K_G) \in \mathbb{Z}^n$ and hence, the integral covering radius of $\mathcal{N}_G$ (with respect to $P_1$) is equal to its covering radius and is hence, lower bounded by $2m/n$. Plugging in $m=n(n-1)/2$ we obtain $2m/n=(n-1) \geq \sqrt{2}\sqrt{g}$. Combined with Corollary \ref{degplus_corr}, this completes the proof of Theorem \ref{bnKn_theo}.  \qed



\appendix
\section{Riemann-Roch Theory for $\mathbb{R}$-Divisors on a Graph}\label{rrreal_sect}

We elaborate on \cite[Section 8.1]{AmiMan10} and state a Riemann-Roch theorem for $\mathbb{R}$-divisors on a graph.  This is a special case, in a slightly different language, of the Riemann-Roch theorem for edge-weighted graphs due to James and Miranda \cite{JamMir13} with each edge weight equal to the number of multiedges between the corresponding pair of vertices.  

 Let $G$ be an undirected, connected multigraph, with no loops, having $n$ vertices, $m$ edges and genus $g$.  Following \cite{AmiMan10}, we define the sigma region of $G$ as follows:
 
 \begin{definition}
 The sigma region $\tilde{\Sigma}^c(G)$ is the closure (under the Euclidean topology) of the subset of ${\rm Div}_{\mathbb{R}}(G)$ consisting of all $\mathbb{R}$-divisors $D$ whose modified rank $\tilde{r}(D)$ is equal to $-1$.
\end{definition}

The set $\tilde{\Sigma}^c(G)$ relates to the set  $\Sigma^c(G)$ introduced just after \cite[Definition 2.3]{AmiMan10} as  $\tilde{\Sigma}^c(G)=-\Sigma^c(G)$, this follows from their definitions.
Let $\widetilde{{\rm Ext}}^c(G)$ be the set of local maxima of the degree function restricted to $\tilde{\Sigma}^c(G)$. By \cite[Item i, Theorem 6.9]{AmiMan10}, the set $\widetilde{{\rm Ext}}^c(G)$ is equal to $\mathcal{N}_G+\sum_v (v)$ where $\mathcal{N}_G$ is the set of non-special divisors on $G$.  Recall that, by definition, $\mathcal{N}_G$ is the set of divisors of degree $g-1$ and rank minus one. The following characterisation of the sigma region is the key to the Riemann-Roch theorem:

\begin{proposition}\label{sigmachar_prop}
An  $\mathbb{R}$-divisor $D$ is contained in the sigma region if and only if there exists an element $\tilde{\nu} \in \widetilde{{\rm Ext}}^c(G)$ such that $\tilde{\nu}-D$ is an effective $\mathbb{R}$-divisor. 
\end{proposition}

Proposition \ref{sigmachar_prop} is a generalisation of \cite[Theorem 2.6]{AmiMan10} to $\mathbb{R}$-divisors. This leads to the following formula for the modified rank of an $\mathbb{R}$-divisor:

\begin{proposition}\label{rankrealdiv_prop}
The modified rank $\tilde{r}(D)$ of an  $\mathbb{R}$-divisor on $G$ is equal to
\begin{center} ${\rm min}_{\tilde{\nu} \in \widetilde{{\rm Ext}}^c(G)}{\rm deg^{+}}(D-\tilde{\nu})-1$. \end{center}
\end{proposition}

Proposition \ref{rankrealdiv_prop} is a generalisation of \cite[Lemma 2.9]{AmiMan10} to $\mathbb{R}$-divisors. However, note the changes in sign in both these generalisations (Propositions \ref{sigmachar_prop} and \ref{rankrealdiv_prop}) that arise from the difference between $\tilde{\Sigma}^c(G)$ and $\Sigma^c(G)$.  As a corollary, we deduce the continuity of the modified rank function $\tilde{r}$.

\begin{corollary}\label{rankcont_cor}
The modified rank ${\tilde{r}}: {\rm Div}_{\mathbb{R}}(G) \rightarrow \mathbb{R}$ is a continuous function where ${\rm Div}_{\mathbb{R}}(G)$ is identified with the real vector space of rank $n$ equipped with the Euclidean topology.
\end{corollary}

The following analogue of the Riemann-Roch theorem follows from Proposition \ref{rankrealdiv_prop}, 

\begin{theorem}{\rm ({\bf Riemann-Roch for $\mathbb{R}$-Divisors: Version 1})}\label{rrreal_theover1}
Let $\tilde{K_G}=\sum_v {\rm val}(v)(v)$ where ${\rm val}(v)$ is the valency of the vertex $v$. Let $g_{\mathbb{R}}=m+1$, where $m$ is the number of edges of $G$. For any  $\mathbb{R}$-divisor $D$,  the following formula holds:  

\begin{center}
$\tilde{r}(D)-\tilde{r}(\tilde{K_G}-D)={\rm deg}(D)-(g_{\mathbb{R}}-1)$.
\end{center}
\end{theorem}

The modified rank $\tilde{r}$ is related to the Baker-Norine rank $r(D)$ in the case where $D \in {\rm Div_{\mathbb{Z}}}(G)$ as $\tilde{r}(D+\sum_v (v))=r(D)$.  Note that the rank $r:{\rm Div}_{\mathbb{R}}(G) \rightarrow \mathbb{R}$ is $r(D):=\tilde{r}(D+\sum_v (v))={\rm min}_{\nu \in \mathcal{N}_G}{\rm deg^{+}}(D-\nu)-1$ (cf. Proposition \ref{degplusrank_theo}). Substituting $D+\sum_{v} (v)$ in Theorem \ref{rrreal_theover1} gives the following version.

\begin{theorem}{\rm ({\bf Riemann-Roch for $\mathbb{R}$-Divisors: Version 2})}\label{rrreal_theover2}
Let $K_G=\sum_v ({\rm val}(v)-2)(v)$ where $v$ is the valency of the vertex $v$ and let $g$ be the genus of $G$. For any  $\mathbb{R}$-divisor $D$,  the following formula holds:  

\begin{center}
$r(D)-r(K_G-D)={\rm deg}(D)-(g-1)$.
\end{center}

\end{theorem}

The proofs of these statements are along the same lines as the proofs of the Riemann-Roch theorem for graphs in \cite{BakNor07} and \cite{AmiMan10}. 

\begin{remark}\label{ranksub_rem} \rm We point out some properties (and related subtleties) of the rank function $r$ (in its extension to $\mathbb{R}$-divisors). 

\begin{enumerate}
\item As in the case of divisors, the rank of an effective $\mathbb{R}$-divisor $E_{\mathbb{R}}$ is non-negative.  This follows from the observation that any element in $\mathcal{N}_G$ is a divisor with rank minus one. Hence, for any element in $\mathcal{N}_G$ there is a vertex where it has coefficient at most minus one.  We then apply the degree plus formula (Proposition \ref{degplusrank_theo}) to conclude that $r(E_{\mathbb{R}}) \geq 0$. A similar argument also shows that an $\mathbb{R}$-divisor of the form $r_0 \cdot (1,\dots,1)$ for a non-negative real number $r_0$ has rank at least $r_0$. 

\item\label{second_item} Unlike in the case of divisors, the rank of an $\mathbb{R}$-divisor of negative degree need not be negative. As an example, consider a tree on $n$ vertices where $n \geq 5$ is an odd integer. In this case, the set $\mathcal{N}_G$ is the set of all divisors of degree minus one. The subset of $\mathcal{N}_G$ that realises the rank of the divisor zero (in the degree plus formula) is precisely $\{-e_i|~$for $i$ from one to $n$ where $e_i$ is the $i$-th standard basis vector of $\mathbb{R}^n\}$. By the discreteness of $\mathcal{N}_G$, the same set also realises the rank of the $\mathbb{R}$-divisor $D_{\mathbb{R}}=\sum_{i=1}^{(n-1)/2} \epsilon \cdot e_i-\sum_{i=(n+1)/2}^{n} \epsilon \cdot e_i$ for sufficiently small $\epsilon>0$.  Its degree is $-\epsilon<0$ and its rank is $(n-3)/2 \cdot \epsilon>0$ (since $n \geq 5$). By Riemann-Roch, this also implies that $K_G-D_{\mathbb{R}}$ has degree strictly greater than $2g-2$ and rank strictly greater than its degree minus $g$. This is also in contrast with the case of divisors. 

\item\label{third_item} The modified rank function $\tilde{r}$, however, behaves similar to the case of divisors in this respect: $\mathbb{R}$-divisors of negative degree have negative modified rank and $\mathbb{R}$-divisors of degree $d$ strictly greater than $2g_{\mathbb{R}}-2$ have modified rank $d-g_{\mathbb{R}}$. This implies that $\mathbb{R}$-divisors of degree strictly less than $-n$ have negative rank and $\mathbb{R}$-divisors of degree $d$ strictly greater than $2g-2+n$ have rank equal to $d-g$.

\item There are $\mathbb{R}$-divisors of degree zero that are not principal and have positive rank: consider a tree on $n$ vertices where $n$ is even, the $\mathbb{R}$-divisor $\sum_{i=1}^{n/2} \epsilon \cdot e_i-\sum_{i=n/2+1}^{n} \epsilon \cdot e_i$ for sufficiently small $\epsilon>0$ has degree zero but is not principal and has positive rank by an argument similar to that in Item \ref{second_item}.  \qed
\end{enumerate} 
 \end{remark}

\footnotesize
\noindent {\bf Author's address:}

\smallskip

\noindent Department of Mathematics,\\
Indian Institute of Technology Bombay,\\
Powai, Mumbai,
India 400076.\\

\noindent {\bf Email id:} madhu@math.iitb.ac.in, madhusudan73@gmail.com.\\

\end{document}